\documentclass[10pt]{amsart}

\usepackage{amssymb}
\usepackage{latexsym}
\usepackage{amsmath}
\usepackage{amscd}
\usepackage{euscript}
\usepackage[all]{xy}
\usepackage{float}
\usepackage{rotating}
\usepackage[usenames,dvipsnames]{color}
\definecolor{c1}{rgb}{0.6, 0.0, 0.0 }
\definecolor{c2}{rgb}{0.0, 0.0, 0.6 }
\usepackage[colorlinks=true,linkcolor= c1,citecolor=c2]{hyperref}

\numberwithin{equation}{section}
\floatstyle{plain}
\newfloat{fig}{thp}{lob}
\floatname{fig}{Figure}
\numberwithin{fig}{section}

 \newtheorem{thm}{Theorem}[section]
\newtheorem*{thmA}{Theorem A}
\newtheorem*{thmB}{Theorem B}
\newtheorem*{thmC}{Theorem C}
\newtheorem*{thmD}{Theorem D}

 \newtheorem{cor}[thm]{Corollary}
 \newtheorem{lem}[thm]{Lemma}
 \newtheorem{prop}[thm]{Proposition}
 \theoremstyle{definition}

 \newtheorem{dfn}[thm]{Definition}
 \newtheorem*{dfn*}{Definition}

 \theoremstyle{remark}
 \newtheorem{rem}[thm]{Remark}
 \newtheorem*{ex}{Example}
\numberwithin{equation}{section}

\newcommand\cA{\mathcal A}
\newcommand\cB{\mathcal B}

\newcommand\cG{\mathcal G}
\newcommand\cH{\mathcal H}
\newcommand\cK{\mathcal K}

\newcommand\cL{\mathcal L}

\newcommand\cP{\mathcal P}

\newcommand\NN{\mathbb N}
\newcommand\RR{\mathbb R}

\newcommand\frS{\mathfrak S}
\newcommand\fra{\mathfrak a}

\newcommand\eps{\varepsilon}
\newcommand\dis{\displaystyle}
\newcommand\ov{\overline}
\newcommand\wt{\widetilde}
\newcommand\wh{\widehat}
\newcommand{\defeq}{\mathrel{\mathop:}=}

\newcommand{\defequ}{\mathrel{\mathop:}\hspace*{-0.72ex}&=}

\newcommand\sess{\sigma_{\rm ess}}

\DeclareMathOperator\dom{dom}
\DeclareMathOperator\ran{ran}
\DeclareMathOperator*{\strlim}{s-lim}

\newcommand\void[1]{}

\def\eps{\varepsilon}
\def\sess{\sigma_{\rm ess}}

\newcommand\frq{\mathfrak q}
\newcommand\frs{\mathfrak s}

\def\sS{{\mathfrak S}}

      \def\dC{{\mathbb C}}

   \def\dN{{\mathbb N}}   
      \def\dR{{\mathbb R}}

\def\cA{{\mathcal A}}   \def\cB{{\mathcal B}}   
      
\def\cG{{\mathcal G}}   \def\cH{{\mathcal H}}   
   \def\cK{{\mathcal K}}   \def\cL{{\mathcal L}}
      
\def\cP{{\mathcal P}}      
   \def\cT{{\mathcal T}}

\begin{document}

\title[Schr\"odinger operators with $\delta$ and $\delta'$-potentials]{Schr\"odinger operators
with $\delta$ and $\delta'$-potentials supported on hypersurfaces}

\author[J. Behrndt]{Jussi Behrndt}
\address{Technische Universit\"{a}t Graz,
Institut f\"{u}r Numerische Mathematik, 
Steyrergasse 30, 8010 Graz, Austria}
\email{behrndt@tugraz.at}

\author[M. Langer]{Matthias Langer}
\address{Department of Mathematics and Statistics,
University of Strathclyde, 26 Richmond Street, Glasgow G1 1XH, United Kingdom}
\email{m.langer@strath.ac.uk}

\author[V. Lotoreichik]{Vladimir Lotoreichik}
\address{Technische Universit\"{a}t Graz,
Institut f\"{u}r Numerische Mathematik, 
Steyrergasse 30, 8010 Graz, Austria}
\email{lotoreichik@math.tugraz.at}



\begin{abstract}
Self-adjoint Schr\"odinger operators
with $\delta$ and $\delta'$-potentials supported on a smooth compact hypersurface
are defined explicitly via boundary conditions. The spectral properties of these operators
are investigated, regularity results on the functions in their domains
are obtained, and analogues of the Birman--Schwinger principle and a variant of Krein's formula
are shown. Furthermore, Schatten--von Neumann type estimates for the differences of the powers of
the resolvents of the Schr\"odinger operators
with $\delta$ and $\delta'$-potentials,
and the Schr\"odinger operator without a singular interaction
are proved. An immediate consequence of these estimates is
the existence and completeness of the wave operators of the corresponding scattering systems, as well as the unitary equivalence
of the absolutely continuous parts of the singularly perturbed and unperturbed Schr\"odinger operators. In the proofs of 
our main theorems we make use of
abstract methods from extension theory of symmetric operators, some algebraic considerations and results on elliptic regularity.
\end{abstract}

\maketitle

% *******************************************************************
% *******************************************************************
\section{Introduction}
% *******************************************************************
% *******************************************************************

Schr\"odinger operators with $\delta$ and $\delta'$-potentials supported on
hypersurfaces play an important role in mathematical physics and
have attracted a lot of attention in the recent past; they
are used for the description of quantum particles
interacting with charged hypersurfaces.
In this introduction we first define the differential operators which are studied
in the present paper. Furthermore, we state and explain our main results on the
spectral and scattering properties of these operators in an easily understandable
but mathematically exact form in Theorems A--D below. Although the remaining part
of the paper can be viewed as a proof of these theorems we mention that
Sections~\ref{sec:3} and \ref{sec:4} contain not only slightly generalized versions of Theorems~A--D
but also other results which are of independent interest.

In the following let $\Sigma$ be a compact connected $C^\infty$-hypersurface which separates the
Euclidean space $\dR^n$ into a bounded domain $\Omega_{\rm i}$
and an unbounded domain $\Omega_{\rm e}$
with common boundary $\partial \Omega_{\rm e} = \partial\Omega_{\rm i} = \Sigma$.
Denote by $\delta_\Sigma$ the $\delta$-distribution  supported on $\Sigma$ and by
$\delta'_\Sigma$ its normal derivative in the distributional sense with the
normal pointing outwards of $\Omega_{\rm i}$.
The main objective of the present paper is to define and study
the spectral properties of Schr\"odinger operators associated with the formal differential expressions
\begin{equation}\label{expressions1}
\cL_{\delta,\alpha} \defeq -\Delta + V - \alpha\, \bigl\langle \delta_\Sigma,\cdot\,\bigr\rangle\,\delta_\Sigma
\quad\text{and}\quad
\cL_{\delta^\prime\!,\beta} \defeq -\Delta + V - \beta\,\bigl \langle\delta_\Sigma',\cdot\,\bigr\rangle\,\delta_\Sigma'.
\end{equation}
Here $V\in L^\infty(\dR^n)$ is assumed to be a real-valued potential and $\alpha,\beta:\Sigma\rightarrow\dR$
are real-valued measurable functions, often called strengths of interactions in mathematical physics.
In order to define the Schr\"odinger operators with $\delta$ and $\delta^\prime$-interactions
rigorously, it is necessary to specify suitable domains in $L^2(\dR^n)$ which take into account the
$\delta$ and $\delta^\prime$-interaction on the hypersurface $\Sigma$. In our approach this
will be done explicitly via suitable interface conditions on $\Sigma$ for a certain
function space in $L^2(\dR^n)$. One of the main advantages of our method
compared with the usual approach via semi-bounded closed sesquilinear forms (see, e.g.\ \cite{BEKS94,E08})
is that $\delta^\prime$-interactions can be treated without any additional difficulties.

Throughout the paper we write the functions $f\in L^2(\dR^n)$
in the form $f=f_{\rm i}\oplus f_{\rm e}$ with respect to the corresponding space
decomposition $L^2(\Omega_{\rm i})\oplus L^2(\Omega_{\rm e})$. For the definition of Schr\"odinger
operators with $\delta$ or $\delta'$-potentials we introduce the following subspaces
\begin{align*}
  H^{3/2}_\Delta(\Omega_{\rm i}) \defequ \bigl\{f_{\rm i}\in H^{3/2}(\Omega_{\rm i})\colon
  \Delta f_{\rm i}\in L^2(\Omega_{\rm i})\bigr\},\\[1ex]
  H^{3/2}_\Delta(\Omega_{\rm e}) \defequ \bigl\{f_{\rm e}\in H^{3/2}(\Omega_{\rm e})\colon
  \Delta f_{\rm e}\in L^2(\Omega_{\rm e})\bigr\},
\end{align*}
of the Sobolev spaces $H^{3/2}(\Omega_{\rm i})$ and $H^{3/2}(\Omega_{\rm e})$, respectively,
and their orthogonal sum in $L^2(\dR^n)$:
\begin{equation*}
  H^{3/2}_\Delta(\dR^n\backslash\Sigma) \defeq H^{3/2}_\Delta(\Omega_{\rm i})
  \oplus H^{3/2}_\Delta(\Omega_{\rm e});
\end{equation*}
cf. \cite{AF03,LM72} and Sections~\ref{sec:fspaces} and \ref{sec:locspaces} for more details.
The trace of a function $f_{\rm i}\in H^{3/2}_\Delta(\Omega_{\rm i})$ and the trace of the normal derivative $\partial_{\nu_{\rm i}}f_{\rm i}$
(with the normal $\nu_{\rm i}$ pointing outwards) are denoted by $f_{\rm i}|_\Sigma$ and
$\partial_{\nu_{\rm i}}f_{\rm i}\vert_\Sigma$, respectively. Similarly, for the exterior domain and $f_{\rm e}\in H^{3/2}_\Delta(\Omega_{\rm e})$
we write $f_{\rm e}|_\Sigma$ and $\partial_{\nu_{\rm e}}f_{\rm e}|_\Sigma$; here $\nu_{\rm e}$ and $\nu_{\rm i}$ are
pointing in opposite directions.

The main objects we study in this paper are the operators given in the following definition,
which are associated with the formal differential expressions in \eqref{expressions1}.

% -------------------------------------------------------------------
\begin{dfn*}
{\it Let $\alpha\in L^\infty(\Sigma)$ be a real-valued function on $\Sigma$.
The Schr\"odinger operator $A_{\delta,\alpha}$
corresponding to the $\delta$-interaction with strength $\alpha$ on $\Sigma$ is defined as
\begin{equation*}
\begin{split}
  A_{\delta,\alpha}f \defequ -\Delta f+Vf, \\
  \dom A_{\delta,\alpha} \defequ \Biggl\{f\in H^{3/2}_\Delta(\dR^n\backslash\Sigma):\begin{matrix}
  f_{\rm i}\vert_\Sigma=f_{\rm e}|_\Sigma\\ 
\alpha f_{\rm i}|_\Sigma=
  \partial_{\nu_{\rm e}}f_{\rm e}|_\Sigma+\partial_{\nu_{\rm i}}f_{\rm i}|_\Sigma\end{matrix} \Biggr\}.
\end{split}
\end{equation*}
Let $\beta$ be a real-valued function on $\Sigma$ such that $1/\beta\in L^\infty(\Sigma)$.
The Schr\"odinger operator $A_{\delta^\prime\!,\beta}$
corresponding to the $\delta^\prime$-interaction with strength $\beta$ on $\Sigma$ is defined as
\begin{equation*}
\begin{split}
  A_{\delta^\prime\!,\beta}f \defequ -\Delta f+ Vf, \\
  \dom A_{\delta^\prime\!,\beta} \defequ \Biggl\{f\in H^{3/2}_\Delta(\dR^n\backslash\Sigma):\begin{matrix}
  \partial_{\nu_e}f_{\rm e}\vert_\Sigma=-\partial_{\nu_{\rm i}}f_{\rm i}\vert_\Sigma\\
  \beta \partial_{\nu_{\rm e}}f_{\rm e}|_\Sigma=f_{\rm e}|_\Sigma-f_{\rm i}|_\Sigma\end{matrix}  \Biggr\}.
\end{split}
\end{equation*}
}
\end{dfn*}

\medskip

The boundary conditions in the domains of $A_{\delta,\alpha}$
and $A_{\delta^\prime\!,\beta}$ fit with the formal differential expressions
in \eqref{expressions1}. In order to see this for $A_{\delta,\alpha}$ we
introduce the closed symmetric form
\begin{equation*}
  \mathfrak a_{\delta,\alpha}[f,g]=(\nabla f,\nabla
g)_{L^2(\dR^n;\dC^n)}+(V f,g)_{L^2(\dR^n)}
  -(\alpha f\vert_\Sigma,g\vert_\Sigma)_{L^2(\Sigma)}
\end{equation*}
on $H^1(\dR^n)$. Further making use of
the boundary conditions $f_{\rm i}\vert_\Sigma=f_{\rm e}|_\Sigma$
and $\alpha f_{\rm i}|_\Sigma=\partial_{\nu_{\rm e}}f_{\rm
e}|_\Sigma+\partial_{\nu_{\rm i}}f_{\rm i}|_\Sigma$
for $f\in\dom A_{\delta,\alpha}$ and the first Green's identity a simple
calculation yields
\begin{equation*}
  (A_{\delta,\alpha} f,g)_{L^2(\dR^n)}=\mathfrak
a_{\delta,\alpha}[f,g]=\bigl\langle\cL_{\delta,\alpha}f,g\bigr\rangle
\end{equation*}
for all $g\in H^1(\dR^n)$. This also shows that $A_{\delta,\alpha}$
coincides with the self-adjoint operator associated
with the closed symmetric form $\mathfrak a_{\delta,\alpha}$; cf.\
Proposition~\ref{prop:exner} for more details.
The quadratic form method has been used
in various papers for the definition of Schr\"odinger operators with
$\delta$-perturbations supported on curves and hypersurfaces. We refer
the reader to \cite{BEKS94}
and the review paper \cite{E08} for more details and further references;
we also mention \cite{BEW09,Ex3,Ex2,EF09,EI01,EK02,EK03,EY02a,KV07} for studies
of eigenvalues,
\cite{BSS00,EF07,EY01,SS01} for results on the absolutely continuous spectrum,
and \cite{AGS87,EK05,EN03,EY02b,EY04,H89,LLP10, S88} for related problems
for Schr\"odinger operators with $\delta$-perturbations.
We point out that the quadratic form approach could not be adapted to
the $\delta^\prime$-case so far; see the open problem posed in \cite[7.2]{E08} and our
solution in Proposition~\ref{prop:deltapform}.
For completeness we also mention that the above definitions of the
differential
operators $A_{\delta,\alpha}$ and $A_{\delta^\prime\!,\beta}$ are compatible
with the ones for one-dimensional $\delta$ and $\delta'$-point interactions
in \cite{AGHH05,AK99}.

In the next theorem, which is the first main result of this paper,
we obtain some basic properties of the Schr\"odinger operators $A_{\delta,\alpha}$ and
$A_{\delta^\prime\!,\beta}$.  Here also the \emph{free} or \emph{unperturbed} Schr\"odinger operator
\begin{equation*}
  A_{\rm free}f =-\Delta f + V f,\qquad \dom A_{\rm free}=H^2(\dR^n),
\end{equation*}
is used. It is well known and easy to see that $A_{\rm free}$
is semi-bounded and self-adjoint in $L^2(\dR^n)$. Recall that the essential spectrum $\sess(A)$
of a self-adjoint operator $A$ consists of all spectral points that are not isolated eigenvalues
of finite multiplicity. The statements in Theorem~A below are contained in
Theorems~\ref{thm:delta}, \ref{thm:delta'}, \ref{thm:smb} and \ref{thm:smb2} in
Sections \ref{sec:delta}--\ref{sec:finiteness}.

% -------------------------------------------------------------------
\begin{thmA}\label{main1}
The Schr\"odinger operators $A_{\delta,\alpha}$ and $A_{\delta^\prime\!,\beta}$
are self-adjoint operators in $L^2(\dR^n)$, which are bounded from below, and their
essential spectra satisfy
\[
  \sigma_{\rm ess}(A_{\delta,\alpha}) = \sigma_{\rm ess}(A_{\delta^\prime\!,\beta})
  = \sigma_{\rm ess}(A_{\rm free}).
\]
If $V \equiv 0$, then $\sess(A_{\delta,\alpha}) = \sess(A_{\delta'\!,\beta}) = [0,\infty)$
and the negative spectra of the self-adjoint operators $A_{\delta,\alpha}$
and $A_{\delta'\!,\beta}$ consist of finitely many negative eigenvalues with finite multiplicities.
\end{thmA}

It is not surprising that additional smoothness assumptions on the functions $\alpha$ and $\beta$ in the boundary condition yield more regularity for the functions
in $\dom A_{\delta,\alpha}$ and $\dom A_{\delta^\prime\!,\beta}$.
The $H^2$-case is of particular importance; see also \cite{BK08}
where the Laplacian on a strip was considered.
The next theorem follows from Theorems~\ref{thm:H2delta} and \ref{thm:H2delta'}.
%Here
%it is convenient to use the notation
%\[
%  H^2(\dR^n\backslash\Sigma) \defeq H^2(\Omega_{\rm i})\oplus H^2(\Omega_{\rm e}).
%\]
The Sobolev space of order one of $L^\infty$-functions on $\Sigma$ is denoted by
$W^{1,\infty}(\Sigma)$.

% -------------------------------------------------------------------
\begin{thmB}
If $\alpha\in W^{1,\infty}(\Sigma)$, then $\dom A_{\delta,\alpha}$ is
contained in $H^2(\Omega_{\rm i})\oplus H^2(\Omega_{\rm e})$.
If $1/\beta\in W^{1,\infty}(\Sigma)$, then $\dom A_{\delta^\prime\!,\beta}$ is
contained in $H^2(\Omega_{\rm i})\oplus H^2(\Omega_{\rm e})$.
\end{thmB}

The fact that the essential spectra of the operators $A_{\delta,\alpha}$, $A_{\delta^\prime\!,\beta}$
and $A_{\rm free}$ in Theorem~A coincide, follows from the observation that the resolvent
differences of these operators are compact. Roughly speaking this is a consequence of the compactness
of the hypersurface $\Sigma$ and Sobolev embedding theorems. However, as can be expected from
the classical results in \cite{B62} (see also \cite{BLLLP10,BS79,BS80, DS75,G84,G11,G11-2,M10}) more specific
considerations yield more precise Schatten--von Neumann type estimates
for the differences of the resolvents and their integer powers, which then in turn imply existence and
completeness of the wave operators of the scattering pairs $\{A_{\delta,\alpha},A_{\rm free}\}$
and $\{A_{\delta'\!,\beta},A_{\rm free}\}$; see, e.g.\ \cite{K95,RS79-III, Y92} for more
details and consequences.

Recall that a compact operator $T$ is said to belong to the weak Schatten--von Neumann ideal $\sS_{p,\infty}$ if
the sequence of singular values $s_k$, i.e.\ the sequence of eigenvalues of the non-negative
operator $(T^*T)^{1/2}$, satisfies $s_k=O(k^{-1/p})$, $k\rightarrow\infty$. 
Note that $\sS_{p,\infty}\subset\sS_{p^\prime}$ for all $p^\prime>p$, where $\sS_{p^\prime}$ is the usual Schatten--von Neumann
ideal; cf. Section~\ref{sec:SvN}.

% -------------------------------------------------------------------
\begin{thmC}
For the self-adjoint Schr\"odinger operators $A_{\delta,\alpha}$ and $A_{\delta'\!,\beta}$
in $L^2(\dR^n)$ the following statements hold.
\begin{itemize}\setlength{\itemsep}{1.2ex}
\item [{\rm (i)}]
For all $\lambda\in\rho(A_{\delta,\alpha})\cap\rho(A_{\rm free})$ we have
\begin{equation*}
  (A_{\delta,\alpha}-\lambda)^{-1} - (A_{\rm free}-\lambda)^{-1} \in \sS_{\frac{n-1}{3},\infty}
\end{equation*}
and, in particular, the wave operators for the pair $\{A_{\delta,\alpha},A_{\rm free}\}$
exist and are complete when $n=2$ or $n=3$.
\item [{\rm (ii)}]
For all $\lambda\in\rho(A_{\delta^\prime\!,\beta})\cap\rho(A_{\rm free})$ we have
\begin{equation*}
  (A_{\delta^\prime\!,\beta}-\lambda)^{-1} - (A_{\rm free}-\lambda)^{-1} \in \sS_{\frac{n-1}{2},\infty},
\end{equation*}
and, in particular, the wave operators for the pair $\{A_{\delta^\prime\!,\beta},A_{\rm free}\}$
exist and are complete when $n=2$.
\end{itemize}
\end{thmC}

The scattering theory for operators with $\delta$-potentials in the
two-dimen\-sional case is partially developed in \cite{EK05}.
In higher dimensions it is necessary to extend the estimates
to higher powers of resolvents as we do in the next main theorem
under an additional local regularity assumption on the potential $V$.
In particular, for sufficiently smooth $V$ this implies the existence
and completeness of the wave operators for the scattering
pairs $\{A_{\delta,\alpha},A_{\rm free}\}$ and $\{A_{\delta^\prime\!,\beta},A_{\rm free}\}$
in any space dimension. For $k\in\dN_0$ the subspace of $L^\infty(\dR^n)$
which consists of
all functions that admit partial derivatives in an open neighbourhood of the
hypersurface $\Sigma$ up to order $k$ in $L^\infty(\dR^n)$
is denoted by $W^{k,\infty}_{\Sigma}(\dR^n)$.

% -------------------------------------------------------------------
\begin{thmD}
Let the self-adjoint Schr\"odinger operators $A_{\delta,\alpha}$ and $A_{\delta'\!,\beta}$
be as above, and assume that $V\in W^{2m-2, \infty}_{\Sigma}(\dR^n)$ for some $m\in\dN$. Then the following statements hold.
\begin{itemize}\setlength{\itemsep}{1.2ex}
\item [{\rm (i)}]
For all $l=1,2,\dots,m$ and $\lambda\in\rho(A_{\delta,\alpha})\cap\rho(A_{\rm free})$ we have
\begin{equation*}
  (A_{\delta,\alpha}-\lambda)^{-l} - (A_{\rm free}-\lambda)^{-l}\in \sS_{\frac{n-1}{2l+1},\infty},
\end{equation*}
and, in particular, the wave operators for the pair $\{A_{\delta,\alpha},A_{\rm free}\}$
exist and are complete when $2m-2 > n-4$.
\item [{\rm (ii)}]
For all $l=1,2,\dots,m$ and $\lambda\in\rho(A_{\delta^\prime\!,\beta})\cap\rho(A_{\rm free})$ we have
\begin{equation*}
  (A_{\delta^\prime\!,\beta}-\lambda)^{-l} - (A_{\rm free}-\lambda)^{-l} \in \sS_{\frac{n-1}{2l},\infty},
\end{equation*}
and, in particular, the wave operators for the pair $\{A_{\delta^\prime\!,\beta},A_{\rm free}\}$
exist and are complete when $2m-2 > n-3$.
\end{itemize}
\end{thmD}

Note that, for $m=1$, Theorem~D reduces to Theorem~C. The proof of Theorem~D is essentially
a consequence of Krein's formula, some algebraic considerations and results on elliptic regularity.
The statements in Theorem~D are contained in Theorems~\ref{thm:S_infty2}, \ref{thm:S_infty4} and
Corollaries~\ref{cor1}, \ref{cor2}.

The paper is organized as follows. Section~\ref{sec:prel} contains
preliminary material on
Schatten--von Neumann classes, general extension theory of symmetric operators and function spaces.
In particular, we prove some useful abstract lemmas on resolvent power differences in Section~\ref{sec:SvN}.
Furthermore,
in Section~\ref{sec:qbt} we collect basic facts about quasi boundary
triples --- a convenient abstract tool from \cite{BL07,BL11} to study self-adjoint extensions
of symmetric partial differential
operators --- and recall a variant of Krein's formula suitable for our purposes.
Section~\ref{sec:3} is devoted to the rigorous mathematical definition and the investigation of the spectral properties of
the operators $A_{\delta,\alpha}$ and $A_{\delta^\prime\!,\beta}$.
In Sections~\ref{sec:delta} and \ref{sec:delta'} we provide
proofs of self-adjointness and sufficient conditions for $H^2$-regularity
of the operator domains, cf.\ Theorems A and B, and we discuss variants of the
Birman--Schwinger principle for the description of eigenvalues of the self-adjoint operators
$A_{\delta,\alpha}$ and $A_{\delta'\!,\beta}$. All these results are obtained by means
of suitable quasi boundary triples constructed in these sections.
%As a by-product of quasi boundary triples we get Krein's formulae exploited
%very much in further sections.
Section~\ref{sec:delta} is accompanied by
a comparison with the sesquilinear form approach to Schr\"odinger operators
with $\delta$-potentials on hypersurfaces. In Section~\ref{sec:finiteness}
we obtain basic spectral properties of the self-adjoint operators $A_{\delta,\alpha}$
and $A_{\delta'\!,\beta}$ such as lower semi-boundedness and finiteness of
negative spectra if $V\equiv 0$. Section~\ref{sec:4} contains our main results on
Schatten--von Neumann estimates from Theorems C and D for resolvent power differences of
operators $A_{\delta,\alpha}$, $A_{\delta'\!,\beta}$ and $A_{\rm free}$.
As a direct consequence of these estimates we establish the existence
and completeness of wave operators for certain scattering pairs arising
in quantum mechanics.

We emphasize again that the results in the body of the paper are sometimes
stronger than in the introduction. Several theorems of their own
independent interest are formulated only in the main part.
We also mention that many of the results in the paper extend  to more general second order
differential operators with sufficiently smooth coefficients and also
remain to be true under weaker assumptions on the smoothness of the hypersurface $\Sigma$; in this context
we refer the reader to the recent papers \cite{AGW10, AGMT10,GM08-1, GM08, GM09, GM11, PR09}
on elliptic operators in non-smooth domains.

% *******************************************************************
% *******************************************************************
\section{Preliminaries}
\label{sec:prel}
% *******************************************************************
% *******************************************************************

This section contains some preliminary material that will be used in the main part of the paper.
In Section~\ref{sec:SvN} we first recall some basic properties of Schatten--von Neumann ideals
and we prove an abstract lemma with sufficient conditions for resolvent power differences
to belong to some Schatten--von Neumann class.
The concept of quasi boundary triples and their Weyl functions from general
extension theory of symmetric operators is briefly reviewed in Section~\ref{sec:qbt}.
Sections~\ref{sec:fspaces} and~\ref{sec:locspaces} contain mainly definitions and
notations for the function spaces used in the paper.

% *******************************************************************
\subsection{$\sS_p$ and $\sS_{p,\infty}$-classes}
\label{sec:SvN}
% *******************************************************************

Let $\cH$ and $\cG$ be separable Hilbert spaces. The space of bounded everywhere defined
linear operators from $\cH$ into $\cG$ is denoted by $\cB(\cH,\cG)$, and we
set $\cB(\cH)\defeq\cB(\cH,\cH)$. The ideal of compact operators mapping
from $\cH$ into $\cG$ is denoted by $\sS_\infty(\cH,\cG)$, and we set
$\sS_\infty(\cH) \defeq \sS_\infty(\cH,\cH)$.
We agree to write $\sS_\infty$ when it is clear from the
context between which spaces the operators act.
The \emph{singular values} (or \emph{$s$-numbers}) $s_k(T)$, $k=1,2,\dots$,
of a compact operator $T\in\sS_\infty(\cH,\cG)$ are defined as the
eigenvalues of the non-negative compact operator $(T^*T)^{1/2}$,
enumerated in non-increasing order and with multiplicities taken into
account. Recall that the singular values of $T$ and $T^*$ coincide; see, e.g.\ \cite[II.\S2.2]{GK69}.
The Schatten--von Neumann class of operator ideals $\sS_p$ and the
weak Schatten--von Neumann class of operator ideals $\sS_{p,\infty}$ are defined as
\begin{equation*}
  \begin{aligned}
    \sS_{p} \defequ \biggl\{T\in\sS_\infty\colon \sum_{k=1}^\infty
    \bigl(s_k(T)\bigr)^p < \infty
    \biggr\}, \\
    \sS_{p,\infty} \defequ \Bigl\{T\in\sS_\infty\colon s_k(T) =
    O(k^{-1/p}),\,k\to\infty\Bigr\},
  \end{aligned}
  \qquad p > 0;
\end{equation*}
they play an important role later on. We refer the reader to
\cite[III.\S7 and III.\S14]{GK69}, \cite[Chapter 2]{S05}
and to \cite[Chapter~11]{BS87} for a detailed study
of the classes $\sS_p$ and $\sS_{p,\infty}$. If a
compact operator $T\in\sS_\infty(\cH,\cG)$ belongs
to $\sS_p$ or $\sS_{p,\infty}$, then we also write $T\in\sS_p(\cH,\cG)$
or $T\in\sS_{p,\infty}(\cH,\cG)$, respectively, if
the spaces $\cH$ and $\cG$ are important in the context.
Moreover, we set
\begin{equation*}
\begin{split}
  &\sS_{p} \cdot \sS_{q} \defeq \bigl\{ T_1T_2 \colon
  T_1\in \sS_{p}, T_2\in \sS_{q}\bigr\}, \\[1ex]
  &\sS_{p,\infty} \cdot \sS_{q,\infty} \defeq \bigl\{ T_1T_2 \colon
  T_1\in \sS_{p,\infty}, T_2\in \sS_{q,\infty}\bigr\}.
\end{split}
\end{equation*}

\medskip

\noindent

The proof of the next statement can be found in \cite{BS87,GK69} and, e.g. \cite[Lemma~2.3]{BLL10}.

% -------------------------------------------------------------------
\begin{lem}\label{splemma}
Let $p,q,r,s,t>0$. Then
the following statements are true:
\begin{itemize}\setlength{\itemsep}{1.2ex}
\item[(i)] $\dis\sS_p\cdot\sS_q=\sS_r$ and
$\dis\sS_{p,\infty}\cdot\sS_{q,\infty}=\sS_{r,\infty}$ when
$p^{-1}+q^{-1}=r^{-1}$, or, equivalently
$$\dis \sS_{\frac{1}{s}}\cdot\sS_{\frac{1}{t}}=
\sS_{\frac{1}{s+t}}\qquad\text{and}\qquad\dis
\sS_{\frac{1}{s},\infty}\cdot\sS_{\frac{1}{t},\infty}=
\sS_{\frac{1}{s+t},\infty};$$
\item[(ii)] If $T\in\sS_p$, then $T^*\in\sS_p$; if $T\in\sS_{p,\infty}$, then $T^*\in\sS_{p,\infty}$;
\item[(iii)] $\dis \sS_{p} \subset \sS_{p,\infty}$ and
$\sS_{p^\prime,\infty}\subset\sS_p$ for all $p^\prime<p$.
\end{itemize}
\end{lem}

\medskip

\noindent
Let $H$ and $K$ be linear operators in a separable Hilbert space $\cH$
and assume that $\rho(H)\cap\rho(K)\not=\varnothing$.
In order to investigate properties of the difference of the $m$th
powers of the resolvents,
\begin{equation*}\label{resdiffhk}
  (H-\lambda)^{-m}-(K-\lambda)^{-m}, \qquad
  \lambda\in\rho(H)\cap\rho(K),\,\,\,\,m\in\dN,
\end{equation*}
recall that, for two elements $a$ and $b$ of some non-commutative algebra,
the following formula holds:
\begin{equation}
\label{eq:diffpow}
  a^m-b^m = \sum_{k=0}^{m-1} a^{m-k-1}\bigl(a-b\bigr)b^k.
\end{equation}
Substituting $a$ and $b$ by the resolvents of $H$ and $K$, respectively,
and setting
\begin{equation}\label{tmk}
  T_{m,k}(\lambda) \defeq
  (H-\lambda)^{-(m-k-1)}\Bigl((H-\lambda)^{-1}-(K-\lambda)^{-1}\Bigr)(K-\lambda)^{-k}
\end{equation}
for $\lambda\in\rho(H)\cap\rho(K)$, $m\in\dN$ and $k=0,1,\dots,m-1$,
we conclude from \eqref{eq:diffpow} that
\begin{equation}
\label{eq:diffpow2}
  (H-\lambda)^{-m}-(K-\lambda)^{-m}= \sum_{k=0}^{m-1} T_{m,k}(\lambda)
\end{equation}
holds for all $\lambda\in\rho(H)\cap\rho(K)$ and $m\in\dN$. In the next
lemma we show that $(H-\lambda)^{-m}-(K-\lambda)^{-m}$
belongs to $\sS_{p,\infty}$ for all $\lambda\in\rho(H)\cap\rho(K)$ if
all the operators
$T_{m,0}(\lambda_0), T_{m,1}(\lambda_0), \dots,T_{m,m-1}(\lambda_0) $
belong to $\sS_{p,\infty}$ for some $\lambda_0\in\rho(H)\cap\rho(K)$.
In the case $m=1$ the statement is well known.
We note that the statement holds in the same form if the class $\sS_{p,\infty}$ 
is replaced by any operator ideal, e.g.\ $\sS_p$.

% -------------------------------------------------------------------
\begin{lem}\label{resdifflemma2}
Let $H$ and $K$ be linear operators in $\cH$ such that
$\rho(H)\cap\rho(K)\not=\varnothing$.
Moreover, let $p>0$, $m\in\dN$ and $T_{m,k}$ be as in \eqref{tmk}, 
and assume that $T_{m,k}(\lambda_0)\in\sS_{p,\infty}(\cH)$
for some $\lambda_0\in\rho(H)\cap\rho(K)$ and all $k=0,1,\dots,m-1$.
Then
\[
  (H-\lambda)^{-m}-(K-\lambda)^{-m}\in\sS_{p,\infty}(\cH)
\]
for all $\lambda\in\rho(H)\cap\rho(K)$.
\end{lem}

\begin{proof}
For $\lambda\in\rho(H)\cap\rho(K)$ define
\begin{equation}\label{EF}
  E_{\lambda} \defeq I + (\lambda-\lambda_0)(H-\lambda)^{-1}
  \quad\text{and}\quad
  F_{\lambda} \defeq I + (\lambda-\lambda_0)(K-\lambda)^{-1}.
\end{equation}
The resolvent identity implies that
\begin{equation}\label{EHFK1}
  E_{\lambda}(H-\lambda_0)^{-1}
  = (H-\lambda_0)^{-1} + (\lambda-\lambda_0)(H-\lambda)^{-1}(H-\lambda_0)^{-1}
  = (H-\lambda)^{-1}
\end{equation}
and, similarly,
\begin{equation}\label{EHFK2}
  (K-\lambda_0)^{-1}F_{\lambda} = (K-\lambda)^{-1}.
\end{equation}
By induction we obtain
\begin{equation}\label{EHFKl}
  E_{\lambda}^l(H-\lambda_0)^{-l} = (H-\lambda)^{-l} \quad\text{and}\quad
  (K-\lambda_0)^{-l}F_{\lambda}^l = (K-\lambda)^{-l}
\end{equation}
for all $l\in\dN$. Set
$D_1(\lambda)\defeq(H-\lambda)^{-1}-(K-\lambda)^{-1}$,
$\lambda\in\rho(H)\cap\rho(K)$.
Then \eqref{EHFK1}, \eqref{EHFK2} and \eqref{EF} imply that
\begin{equation}\label{resdiffideal2}
\begin{split}
 E_{\lambda}D_1(\lambda_0)F_{\lambda}
  &= E_\lambda(H-\lambda_0)^{-1}F_\lambda - E_\lambda(K-\lambda_0)^{-1}F_\lambda \\
  &= (H-\lambda)^{-1} F_\lambda - E_\lambda (K-\lambda)^{-1} \\
  &= (H-\lambda)^{-1} + (\lambda-\lambda_0)(H-\lambda)^{-1}(K-\lambda)^{-1} \\
  &\quad\qquad - (K-\lambda)^{-1} - (\lambda-\lambda_0)(H-\lambda)^{-1}(K-\lambda)^{-1} \\
  &= D_1(\lambda).
\end{split}
\end{equation}
For $k=0,1\dots,m-1$ and all $\lambda\in\rho(H)\cap\rho(K)$ we obtain from~\eqref{EHFKl},
\eqref{resdiffideal2} and the facts that $E_{\lambda}$ commutes with $(H-\lambda_0)^{-1}$ and
$F_{\lambda}$ commutes with $(K-\lambda_0)^{-1}$ the following relation
\begin{align*}
  T_{m,k}(\lambda)
  &= (H-\lambda)^{-(m-k-1)}D_1(\lambda)(K-\lambda)^{-k}
  \\[1ex]
  &= (H-\lambda)^{-(m-k-1)}E_{\lambda}D_1(\lambda_0)F_{\lambda}(K-\lambda)^{-k}
  \\[1ex]
  &= E_{\lambda}^{m-k-1}(H-\lambda_0)^{-(m-k-1)}E_{\lambda}D_1(\lambda_0)
  F_\lambda(K-\lambda_0)^{-k}F_{\lambda}^{k}
  \\[1ex]
  &= E_{\lambda}^{m-k}(H-\lambda_0)^{-(m-k-1)}D_1(\lambda_0)(K-\lambda_0)^{-k}F_{\lambda}^{k+1}
  \\[1ex]
  &= E_{\lambda}^{m-k}T_{m,k}(\lambda_0)F_{\lambda}^{k+1}.
\end{align*}
By assumption, $T_{m,k}(\lambda_0)\in\sS_{p,\infty}$, and hence we conclude that
$T_{m,k}(\lambda)\in\sS_{p,\infty}$ for $k=0,1,\dots,m-1$ and $\lambda\in\rho(H)\cap\rho(K)$.
This together with~\eqref{eq:diffpow2} implies that
\[
  (H-\lambda)^{-m}-(K-\lambda)^{-m}= \sum_{k=0}^{m-1} T_{m,k}(\lambda)
  \in \sS_{p,\infty}(\cH)
\]
for all $\lambda\in\rho(H)\cap\rho(K)$.
\end{proof}

The following lemma will be used in Section~\ref{sec:pwrdiff} to show that certain
resolvent power differences are in some class $\sS_{p,\infty}$.

%-------------------------------------------------------------------------
\begin{lem}\label{lem:respwrdiff}
Let $H$ and $K$ be linear operators in $\cH$, let $\cK$ be an auxiliary
Hilbert space and assume that, for some $\lambda_0\in\rho(H)\cap\rho(K)$,
there exist operators $B\in\cB(\cK,\cH)$ and $C\in\cB(\cH,\cK)$
such that
\begin{equation}\label{fact}
  (H-\lambda_0)^{-1} - (K-\lambda_0)^{-1} = BC.
\end{equation}
Let $a>0$ and $b_1,b_2\ge0$ be such that $a \le b_1+b_2$ and set
$b \defeq b_1 + b_2-a$.
Moreover, let $r\in\NN$ and assume that
\begin{equation}\label{prodinSr}
  \begin{aligned}
    & (K-\lambda_0)^{-k}B \in \sS_{\frac{1}{ak+b_1},\infty}, \\[0.5ex]
    & C(K-\lambda_0)^{-k} \in \sS_{\frac{1}{ak+b_2},\infty},
  \end{aligned}
  \qquad k=0,1,\dots,r-1.
\end{equation}
Then
\begin{equation}\label{respwrdiff}
  (H-\lambda)^{-l} - (K-\lambda)^{-l} \in \sS_{\frac{1}{al+b},\infty}
\end{equation}
for all $\lambda\in\rho(H)\cap\rho(K)$ and all $l=1,2,\dots,r$.
\end{lem}

\begin{proof}
We prove the statement by induction with respect to $l$.
Using the factorization in \eqref{fact}, the assumptions
in \eqref{prodinSr} with $k=0$ and Lemma~\ref{splemma}\,(i) we obtain
\[
  (H-\lambda_0)^{-1} - (K-\lambda_0)^{-1} = BC \in
  \sS_{\frac{1}{b_1},\infty}\cdot\sS_{\frac{1}{b_2},\infty}
  = \sS_{\frac{1}{b_1+b_2},\infty} = \sS_{\frac{1}{a+b},\infty}.
\]
Now Lemma~\ref{resdifflemma2} with $m=1$ implies that
\[
  (H-\lambda)^{-1} - (K-\lambda)^{-1} \in \sS_{\frac{1}{a+b},\infty}
\]
for all $\lambda\in\rho(H)\cap\rho(K)$, i.e.\ \eqref{respwrdiff} is true
for $l=1$.

For the induction step fix $m\in\NN$, $2\le m \le r$ and assume that
\eqref{respwrdiff}
is satisfied for all $l= 1,2,\dots,m-1$. For $k=0,1,\dots,m-1$ let
$T_{m,k}$ be as in \eqref{tmk}, define
\[
  D_j(\lambda_0) \defeq (H-\lambda_0)^{-j}-(K-\lambda_0)^{-j},\qquad j\in\dN_0,
\]
and write
\begin{equation}
  \begin{aligned}
     &T_{m,k}(\lambda_0) =
    (H-\lambda_0)^{-(m-k-1)}BC(K-\lambda_0)^{-k}\\
    &\qquad\qquad\qquad\qquad= D_{m-k-1}(\lambda_0)BC(K-\lambda_0)^{-k}\\
   &\qquad\qquad\qquad\qquad\qquad\qquad + (K-\lambda_0)^{-(m-k-1)}BC(K-\lambda_0)^{-k}.
\label{two_summ}
  \end{aligned}
\end{equation}
Note that $D_0(\lambda_0) = 0$.  By assumption \eqref{prodinSr} we have
\begin{equation*}
\begin{split}
&  B \in \sS_{\frac{1}{b_1},\infty}, \!\quad C(K-\lambda_0)^{-k}
  \in \sS_{\frac{1}{ak+b_2},\infty},\\
 & (K-\lambda_0)^{-(m-k-1)}B \in \sS_{\frac{1}{a(m-k-1)+b_1},\infty},
\end{split}
\end{equation*}
for $k=0,1,\dots,m-1$.
By the induction assumption we also have
\[
  D_{m-k-1}(\lambda_0) \in \sS_{\frac{1}{a(m-k-1)+b},\infty}
\]
for $k=0,1,\dots,m-1$, and and hence we obtain with Lemma~\ref{splemma}\,(i)
that the first summand in \eqref{two_summ} is in
\[
  \sS_{\frac{1}{a(m-k-1)+b},\infty}\cdot\sS_{\frac{1}{b_1},\infty}\cdot\sS_{\frac{1}{ak+b_2},\infty}
  = \sS_{\frac{1}{am+2b},\infty} \subset \sS_{\frac{1}{am+b},\infty},
\]
where we used that $b\ge0$. 
The second summand in \eqref{two_summ} is in
\[
  \sS_{\frac{1}{a(m-k-1)+b_1},\infty}\cdot\sS_{\frac{1}{ak+b_2},\infty}
  = \sS_{\frac{1}{am+b},\infty}.
\]
Hence $T_{m,k}(\lambda_0) \in \sS_{\frac{1}{am+b},\infty}$ for all
$k=0,1,\dots,m-1$.  Now Lemma~\ref{resdifflemma2} implies the validity
of \eqref{respwrdiff} for $l=m$.
\end{proof}

% *******************************************************************
\subsection{Quasi boundary triples and their Weyl functions}
\label{sec:qbt}
% *******************************************************************

The concept of quasi boundary triples and Weyl functions is a
generalization of the notion of (ordinary) boundary triples and
Weyl functions from \cite{Br76,DM91,GG91,Ko75}, which is a very convenient tool in extension theory of symmetric operators. 
Quasi boundary triples are particularly useful
when dealing with elliptic boundary value problems from an operator and
extension theoretic point of view.
In this subsection we provide some general facts on quasi boundary
triples which can be found in \cite{BL07} and \cite{BL11}.

Throughout this subsection let $(\cH,(\cdot,\cdot)_\cH)$ be a Hilbert
space and let $A$ be a densely defined closed symmetric operator in $\cH$.

% -------------------------------------------------------------------
\begin{dfn}
\label{def:qbt}
A triple $\{\cG,\Gamma_0,\Gamma_1\}$ is called a \emph{quasi boundary
triple}
for $A^*$ if $(\cG,(\cdot,\cdot)_\cG)$ is a Hilbert space and for some
linear operator $T\subset A^*$ with $\ov T  = A^*$ the following holds:
\begin{itemize}\setlength{\itemsep}{1.2ex}
\item [{\rm (i)}] $\Gamma_0,\Gamma_1:\dom T\rightarrow\cG$ are linear
mappings and $\ran\binom{\Gamma_0}{\Gamma_1}$ is dense in $\cG\times\cG$;
\item [{\rm (ii)}]  $A_0 \defeq T\upharpoonright\ker\Gamma_0$ is a self-adjoint
operator in $\cH$;
\item [{\rm (iii)}] for all $f,g\in \dom T$ the {\em abstract Green's
identity} holds:
\begin{equation}\label{green11}
  (Tf,g)_{\cH}-(f,Tg)_{\cH}
  =(\Gamma_1 f,\Gamma_0 g)_{\cG}-(\Gamma_0 f,\Gamma_1 g)_{\cG}.
\end{equation}
\end{itemize}
\end{dfn}

\medskip

\noindent

The following simple example illustrates the notion of quasi boundary triples 
for the Laplacian on a smooth bounded domain, see \cite{BL07,BL11}, Section~\ref{sec:notations} and 
Proposition~\ref{prop:qbt0}.

\begin{ex}
Let $\Omega$ be a bounded domain with smooth boundary, $A=-\Delta$
with $\dom A=H^2_0(\Omega)$, $T=-\Delta$ with $\dom T=H^2(\Omega)$,
let $\cG=L^2(\partial\Omega)$ and define the boundary mappings as
\[
  \Gamma_0f = \partial_{\nu}f|_{\partial\Omega}, \quad
  \Gamma_1f = f|_{\partial\Omega}, \qquad
  f\in\dom T;
\]
where $\partial_{\nu}$ stands for the normal derivative with normal vector pointing outwards.
It can be shown that the closure of $T$ coincides with the adjoint operator $A^*=-\Delta$, $\dom A^*=\{f\in L^2(\Omega):\Delta f\in L^2(\Omega)\}$,
and that the properties of (i)-(iii) in Definition~\ref{def:qbt} hold. Hence $\{L^2(\partial\Omega),\Gamma_0,\Gamma_1\}$ is a quasi boundary triple
for $A^*$.
\end{ex}

We remark that a quasi boundary triple for the adjoint $A^*$ of a densely defined
closed symmetric operator exists if and only if
the deficiency indices $n_\pm(A)=\dim\ker(A^*\mp i)$ of $A$ coincide.
Moreover, if  $\{\cG, \Gamma_0, \Gamma_1 \}$ is a quasi boundary triple for $A^*$,
then $A$ coincides with $T\upharpoonright(\ker\Gamma_0\cap\ker\Gamma_1)$
and the operator $A_1 \defeq T\upharpoonright\ker\Gamma_1$ is symmetric in $\cH$.
We also mention that a quasi-boundary triple with the additional property $\ran\Gamma_0 =\cG$
is a generalized boundary triple in the sense of~\cite{DHMS06,DM95}. In the special case
that the deficiency indices $n_\pm(A)$ of $A$ are finite (and coincide) a quasi
boundary triple is automatically an ordinary boundary triple.

The following proposition contains a sufficient condition for a
triple $\{\cG,\Gamma_0,\Gamma_1\}$ to be a quasi boundary triple, cf.\ \cite[Theorem~2.3]{BL07}
and \cite[Theorem~2.3]{BL11}.
The result will be applied in Sections~\ref{sec:delta} and \ref{sec:delta'}.

% -------------------------------------------------------------------
\begin{prop}\label{suff_cond_qbt}
Let $\cH$ and $\cG$ be Hilbert spaces and let $T$ be a linear operator in $\cH$.
Assume that $\Gamma_0,\Gamma_1\colon \dom T\rightarrow\cG$ are linear mappings such
that the following conditions are satisfied:
\begin{itemize}\setlength{\itemsep}{1.2ex}
\item[(a)]
The range of
$\,\binom{\Gamma_0}{\Gamma_1}: \dom T\rightarrow\cG\times\cG$
is dense and $\ker\Gamma_0\cap\ker\Gamma_1$ is dense in $\cH$.
\item[(b)]
The identity \eqref{green11} holds for all $ f,g\in \dom T$.
\item[(c)] $T\upharpoonright {\ker\Gamma_0}$ is an extension of a self-adjoint
operator $A_0$.
\end{itemize}
Then $A \defeq T\upharpoonright {\ker\Gamma_0\cap\ker\Gamma_1}$ is a densely defined closed
symmetric operator in $\cH$, and $\{\cG,\Gamma_0,\Gamma_1\}$ is a
quasi boundary triple for $A^*$ with $A_0=T\upharpoonright \ker\Gamma_0$.
\end{prop}

Next we recall the definition of the $\gamma$-field and the Weyl
function associated with a quasi boundary triple $\{\cG, \Gamma_0,
\Gamma_1 \}$
for $A^*$.
Note first that the decomposition
\begin{equation*}
\dom T=\dom A_0\,\dot +\,\ker(T-\lambda)=\ker\Gamma_0\,\dot
+\,\ker(T-\lambda)
\end{equation*}
holds for all $\lambda\in\rho(A_0)$. Hence $\Gamma_0\upharpoonright\ker(T-\lambda)$
is invertible for all $\lambda\in\rho(A_0)$ and maps $\ker(T-\lambda)$ bijectively
onto $\ran\Gamma_0$.

\begin{dfn}\label{def:gammaWeyl}
Let $\{\cG,\Gamma_0,\Gamma_1\}$ be a quasi boundary triple for
$\overline T= A^*$ and
$A_0 = T\upharpoonright\ker\Gamma_0$.
Then the (operator-valued) functions $\gamma$ and $M$ defined by
\begin{equation*}\label{gweyl}
  \gamma(\lambda) \defeq
  \bigl(\Gamma_0\upharpoonright\ker(T-\lambda)\bigr)^{-1}
  \quad\text{and}\quad
  M(\lambda) \defeq \Gamma_1\gamma(\lambda),\quad
  \lambda\in\rho(A_0),
\end{equation*}
are called the $\gamma$\emph{-field} and the \emph{Weyl function}
corresponding to the quasi boundary triple $\{\cG,\Gamma_0,\Gamma_1\}$.
\end{dfn}

The values of the Weyl function corresponding to the quasi boundary triple $\{L^2(\partial\Omega),\Gamma_0,\Gamma_1\}$ in the
example below Definition~\ref{def:qbt} are Neumann-to-Dirichlet maps; cf. \cite{BL07,BL11}, Section~\ref{sec:notations} and 
Proposition~\ref{prop:qbt0}.

%In the example above where $T=-\Delta$ the Weyl function coincides with
%the Neumann-to-Dirichlet map.

The definitions of $\gamma$ and $M$ coincide with the definitions of
the $\gamma$-field and the Weyl function in the case
that $\{\cG,\Gamma_0,\Gamma_1\}$ is an
ordinary boundary triple, cf.\ \cite{DM91}.
Note that, for each $\lambda \in \rho(A_0)$, the operator
$\gamma(\lambda)$ maps $\ran\Gamma_0$ into $\cH$ and $M(\lambda)$ maps
$\ran\Gamma_0$ into $\ran\Gamma_1$. Furthermore, as an immediate consequence of the
definition of $M(\lambda)$ we obtain
\begin{align*}\label{NDpropertyAbstr}
  M(\lambda) \Gamma_0 f_\lambda = \Gamma_1 f_\lambda, \qquad f_\lambda
  \in \ker (T - \lambda), \quad \lambda \in \rho(A_0).
\end{align*}

In the next proposition we collect some properties of the $\gamma$-field
and the Weyl function; all statements are proved in \cite{BL07}.

% --------------------------------------------------------------------
\begin{prop} \label{gammaprop}
Let
$\{\cG,\Gamma_0,\Gamma_1\}$ be a quasi boundary triple for $\overline T=A^*$
with $A_0=T\upharpoonright\ker\Gamma_0$, $\gamma$-field $\gamma$ and Weyl
function $M$.
Then, for $\lambda\in \rho (A_0)$, the following assertions hold.
\begin{itemize}\setlength{\itemsep}{1.2ex}
\item[(i)] % -----
$\gamma(\lambda)$ is a densely defined bounded operator from $\cG$ into
$\cH$ with $\dom\gamma(\lambda)=\ran\Gamma_0$.
\item[(ii)] % -----
The adjoint of $\gamma(\ov\lambda)$ satisfies
\begin{equation*}
\gamma(\ov\lambda)^* = \Gamma_1(A_0-\lambda)^{-1} \in\cB(\cH,\cG).
\end{equation*}
\item[(iii)] % -----
The values of the Weyl function $M$ are densely defined (in general unbounded) operators in
$\cG$ with $\dom M(\lambda)=\ran\Gamma_0$ and $\ran M(\lambda)\subset\ran\Gamma_1$.
Furthermore, $M(\ov \lambda) \subset M(\lambda)^*$ holds.
\item [(iv)]
If $\ran\Gamma_0 = \cG$, then $M(\lambda)\in\cB(\cG)$.
\item [(v)] If $A_1=T\upharpoonright\ker\Gamma_1$ is a self-adjoint
operator in $\cH$ and $\lambda\in\rho(A_0)\cap\rho(A_1)$, then
$M(\lambda)$ is a bijective mapping from $\ran\Gamma_0$ onto $\ran\Gamma_1$.
\end{itemize}
\end{prop}

\medskip

\noindent
With the help of a quasi boundary triple and the associated Weyl
function it is possible to describe
the spectral properties of extensions of $A$, which are restrictions of
$T\subset A^*$. The extensions $A_\Theta$ are defined with the
help of an abstract boundary condition by
\begin{equation}\label{athetaext}
  A_\Theta \defeq T\upharpoonright\ker\bigl(\Gamma_1-\Theta\Gamma_0\bigr)
  =T\upharpoonright\ker\bigl(\Theta^{-1}\Gamma_1-\Gamma_0\bigr),
\end{equation}
where $\Theta$ is a linear operator in $\cG$ or a linear relation in $\cG$, i.e.\
a subspace of $\cG\times\cG$, cf.\ \cite{BL07}. The sums and products are
understood in the sense of linear relations
if $\Theta$ or $\Theta^{-1}$ is not a (single-valued) operator. However, for our
purposes the case that
$\Theta^{-1}$ is a bounded linear operator on $\cG$ is of particular
interest and linear relations will not be used in the following.
The next statement contains a variant of Krein's formula
in this case; see \cite[Theorem~2.8 and Theorem~4.8]{BL07},
\cite[Theorem~3.7 and Corollary~3.9]{BL11} and
\cite[Theorem~3.13]{BLL10}.

% -------------------------------------------------------------------
\begin{thm}
\label{thm:krein}
Let $\{\cG,\Gamma_0,\Gamma_1\}$
be a quasi boundary triple for $\overline T = A^*$ with
$A_0=T\upharpoonright\ker\Gamma_0$, $\gamma$-field $\gamma$
and Weyl function $M$. Furthermore, let $B=B^*=\Theta^{-1}\in\cB(\cG)$ and let
\begin{equation}\label{atb}
 A_\Theta=T\upharpoonright\ker\bigl(B \Gamma_1-\Gamma_0\bigr)
\end{equation}
be the corresponding extension as in \eqref{athetaext}.
Then, for $\lambda\in \rho (A_0)$, the following assertions hold.
\begin{itemize} \setlength{\itemsep}{1.2ex}
\item [(i)] $\lambda\in\sigma_p(A_\Theta)$ if and only if
$\ker(I-BM(\lambda))\not=\{0\}$.
Moreover, in this case, the multiplicity of the eigenvalue $\lambda$ of $A_\Theta$ is equal
to $\dim\ker(I-BM(\lambda))$.
\item [(ii)] For all $g\in\ran(A_\Theta-\lambda)$ and $\lambda\notin\sigma_p(A_\Theta)$ we have
\begin{equation*}
  (A_\Theta-\lambda)^{-1}g- (A_0-\lambda)^{-1}g = \gamma(\lambda)
  \bigl(I- BM(\lambda)\bigr)^{-1}B\gamma(\ov\lambda)^*g.
\end{equation*}
\end{itemize}
If, in addition, $\ran \Gamma_0 = \cG$ and
$M(\lambda_0)\in\sS_\infty(\cG)$ for some $\lambda_0\in\dC\setminus\dR$,
then the operator $A_\Theta$ in \eqref{atb} is self-adjoint in $\cH$, Krein's formula
\begin{equation}\label{krein1}
  (A_\Theta-\lambda)^{-1}- (A_0-\lambda)^{-1} = \gamma(\lambda)
  \bigl(I-BM(\lambda)\bigr)^{-1}B\gamma(\ov\lambda)^*
\end{equation}
holds for all $\lambda\in\rho(A_\Theta)\cap\rho(A_0)$,
and $(I- BM(\lambda))^{-1}\in\cB(\cG)$.
\end{thm}

% *******************************************************************
\subsection{Sobolev spaces, traces and Green's identities}
\label{sec:fspaces}
% *******************************************************************

Throughout this paper Sobolev spaces and certain interpolation spaces
play an important role. In this subsection we provide some necessary
definitions and basic properties. The reader is referred, e.g.\ to the
monographs \cite{AF03,GBook08,LM72,McL00} for more details.

Let $\Omega \subset \dR^n$ be an arbitrary bounded or unbounded domain
with a compact $C^\infty$-boundary $\partial\Omega)$.
By $H^s(\Omega)$ and $H^s(\partial\Omega)$, $s\in\RR$, we denote the
standard ($L^2$-based) Sobolev spaces of order $s$ of functions in $\Omega$
and $\partial\Omega$, respectively. The inner product and norm on $H^s$
are denoted by $(\cdot,\cdot)_s$ and $\Vert\cdot\Vert_s$, for $s=0$
we simply write $(\cdot,\cdot)$ and $\Vert\cdot\Vert$, respectively.
In order to avoid possible confusion, sometimes also the space is used as
an index, e.g.\ $(\cdot,\cdot)_{L^2(\Omega)}$ and
$(\cdot,\cdot)_{L^2(\partial\Omega)}$. The Sobolev spaces
of order $k\in\dN_0$ of $L^\infty$-functions on $\Omega$
and $\partial\Omega$ are denoted by
$W^{k,\infty}(\Omega)$ and $W^{k,\infty}(\partial\Omega)$, respectively.
The following well-known implications will be used later:
\begin{equation}
\begin{alignedat}{2}
\label{W1Hk}
  & f \in H^k(\Omega),\, g\in W^{k,\infty}(\Omega) &\quad
  &\Longrightarrow\quad fg\in H^k(\Omega),\qquad k\in\dN_0; \\[0.5ex]
  & h \in H^1(\partial\Omega),\, k\in W^{1,\infty}(\partial\Omega) &\quad
  &\Longrightarrow \quad hk\in H^1(\partial\Omega).
\end{alignedat}
\end{equation}

For a function $f$ on $\Omega$ we denote by $f\vert_{\partial \Omega}$
and $\partial_\nu f|_{\partial\Omega}$
the trace and the trace of the normal derivative (with normal vector pointing outwards),
respectively. For $s > 3/2$ the trace mapping
\begin{equation}
\label{Sobtrace}
  H^s(\Omega) \ni f\mapsto \bigl\{f|_{\partial \Omega}, \partial_\nu
  f|_{\partial\Omega}\bigr\}
  \in H^{s-1/2}(\partial\Omega)\times H^{s-3/2}(\partial\Omega)
\end{equation}
is the continuous extension of the trace mapping defined on
$C^\infty$-functions. Recall that for $s > 3/2$ the mapping~\eqref{Sobtrace}
is surjective onto $H^{s-1/2}(\partial\Omega)\times
H^{s-3/2}(\partial\Omega)$.

Besides the Sobolev spaces $H^s(\Omega)$ the spaces
\begin{equation*}
  H^s_\Delta(\Omega) \defeq \bigl\{f\in H^s(\Omega)\colon \Delta f\in
  L^2(\Omega)\bigr\}, \quad s\ge 0,
\end{equation*}
equipped with the inner product $(\cdot,\cdot)_s+
(\Delta\,\cdot\,,\Delta\,\cdot)$ and corresponding norm will be useful.
Observe that
for $s \ge  2$ the spaces $H^s_\Delta(\Omega)$ and $H^s(\Omega)$
coincide. We also note that $H^s_\Delta(\Omega)$, $s\in (0,2)$, can
be viewed as an interpolation space between $H^2(\Omega)$ and
$H^0_\Delta(\Omega)$, where the latter space coincides with
the maximal domain of the Laplacian in $L^2(\Omega)$. By \cite{LM72} the
trace mapping can be extended to a continuous mapping
\begin{equation}
\label{LMtrace}
  H^s_\Delta(\Omega) \ni f\mapsto \bigl\{f|_{\partial\Omega},
  \partial_\nu f|_{\partial\Omega} \bigr\}\in
  H^{s-1/2}(\partial\Omega)\times H^{s-3/2}(\partial\Omega)
\end{equation}
for all $s \in [0,2)$, where each of  the mappings
\begin{equation*}
\begin{split}
  &H^s_{\Delta}(\Omega) \ni f\mapsto f|_{\partial\Omega}\in
  H^{s-1/2}(\partial\Omega), \\
  &H^s_\Delta(\Omega) \ni f\mapsto \partial_\nu f|_{\partial\Omega}\in
  H^{s-3/2}(\partial\Omega)
\end{split}
\end{equation*}
is surjective for $s\in [0,2)$.
We also recall
that the first and second Green's identities hold for all $f,g\in H^{3/2}_\Delta(\Omega)$
and $h\in H^1(\Omega)$:
\begin{equation}
\label{half_green}
  \bigl(-\Delta f,h\bigr)_{L^2(\Omega)}  = \bigl(\nabla f,\nabla
  h\bigr)_{L^2(\Omega;\dC^n)} -
  \bigl(\partial_\nu f|_{\partial\Omega},
  h|_{\partial\Omega}\bigr)_{L^2(\partial\Omega)}
\end{equation}
and
\begin{equation}
\label{green_identity}
\begin{split}
&  \bigl(-\Delta f,g\bigr)_{L^2(\Omega)} - \bigl(f,-\Delta
  g\bigr)_{L^2(\Omega)} \\
&\qquad\qquad\qquad\qquad = \bigl(f|_{\partial\Omega},
  \partial_\nu
  g|_{\partial\Omega}\bigr)_{L^2(\partial\Omega)} - \bigl(\partial_\nu f|_{\partial\Omega},
  g|_{\partial\Omega}\bigr)_{L^2(\partial\Omega)},
\end{split}
\end{equation}
cf.\ \cite{Fr67,LM72} and, e.g.\ \cite[Theorem~4.2]{BLL10}.

% *******************************************************************
\subsection{Some local Sobolev spaces}
\label{sec:locspaces}
% *******************************************************************

Let $\Sigma$ be a compact connected $C^\infty$-hypersurface which
separates the Euclidean space $\dR^n$ into a bounded (interior) domain
$\Omega_{\rm i}$ and an unbounded (exterior) domain $\Omega_{\rm e}$. In
particular, $\Sigma=\partial\Omega_{\rm i}=\partial\Omega_{\rm e}$.
For $s \ge 0$ we use the short notation
\begin{equation}\label{opsops}
  H^s(\dR^n\backslash\Sigma) \defeq H^s(\Omega_{\rm i})\oplus H^s(\Omega_{\rm e})\quad\text{and}\quad
  H^s_\Delta(\dR^n\backslash\Sigma) \defeq H^s_\Delta(\Omega_{\rm i})\oplus H^s_\Delta(\Omega_{\rm e}).
\end{equation}

We denote by $H^s_{\Sigma}(\Omega_{\rm i})$ with $s\ge 0$ the subspace
of $L^2(\Omega_{\rm i})$ which consists of functions that belong to
$H^s$ in a neighbourhood of $\Sigma=\partial\Omega_{\rm i}$, i.e.\
\begin{equation*}
\begin{split}
  H^s_{\Sigma}(\Omega_{\rm i}) \defequ \bigl\{f\in L^2(\Omega_{\rm i})\colon
  \exists~\text{domain}~\Omega'\subset \Omega_{\rm i}~\text{such that}~\\
&\qquad\qquad\qquad\qquad\qquad\qquad  \partial\Omega' \supset
  \Sigma~\text{and}~f\upharpoonright\Omega' \in H^s(\Omega') \bigr\}.
\end{split}
\end{equation*}
The space $H^s_{\Sigma}(\Omega_{\rm e})$ is defined in the same way with
$\Omega_{\rm i}$ replaced by $\Omega_{\rm e}$. The local Sobolev spaces
$H^s_{\Sigma}(\dR^n)$ and $H^s_{\Sigma}(\dR^n\backslash\Sigma)$ in the
next definition consist of $L^2$-functions which are $H^s$ in a
neighbourhood of $\Sigma$ or in both one-sided neighbourhoods of $\Sigma$, respectively.

% -------------------------------------------------------------------
\begin{dfn}\label{deflocsob}
Let $\Sigma$, $\Omega_{\rm i}$, $\Omega_{\rm e}$, and the spaces
$H^s_{\Sigma}(\Omega_{\rm i})$ and $H^s_{\Sigma}(\Omega_{\rm e})$,
$s\geq 0$,
be as above. Then we define
\begin{equation*}
\begin{split}
 H^s_{\Sigma}(\dR^n) \defequ \big\{f\in L^2(\dR^n)\colon
  \exists~\text{domain}~\Omega'\subset \dR^n~\text{such that}\\
&\qquad\qquad\qquad\qquad\qquad\qquad \Omega'\supset
  \Sigma ~\text{and}~f\upharpoonright\Omega' \in H^s(\Omega') \big\}, \\[0.5ex]
 H^s_{\Sigma}(\dR^n\backslash\Sigma)\defequ H^s_\Sigma(\Omega_{\rm i})
  \oplus H^s_\Sigma(\Omega_{\rm e}).
\end{split}
\end{equation*}
\end{dfn}

It follows from the above definition that
$H^s_{\Sigma}(\dR^n)\subsetneq H^s_{\Sigma}(\dR^n\backslash\Sigma)$ holds for all $s>0$.

\medskip

For $k\in\dN_0$ we denote by $W^{k,\infty}_{\Sigma}(\Omega_{\rm i})$ the
subspace of $L^\infty(\Omega_{\rm i})$ which consists of functions that
belong to $W^{k,\infty}$ in a neighbourhood of $\Sigma=\partial\Omega_{\rm i}$, i.e.\
\begin{equation*}
\begin{split}
  W^{k,\infty}_{\Sigma}(\Omega_{\rm i}) \defequ \bigl\{ f\in
  L^\infty(\Omega_{\rm i})\colon
  \exists~\text{domain}~\Omega'\subset\Omega_{\rm i}~\text{such that}\\
&\qquad\qquad\qquad\qquad\qquad\qquad\partial\Omega'
  \supset\Sigma~\text{and}~f\upharpoonright\Omega'\in W^{k,\infty}(\Omega')\bigr\}.
\end{split}
\end{equation*}
The space $W^{k,\infty}_{\Sigma}(\Omega_{\rm e})$ is defined in the same
way with $\Omega_{\rm i}$ replaced by $\Omega_{\rm e}$.
In analogy to Definition~\ref{deflocsob} we introduce the local Sobolev
spaces $W^{k,\infty}_{\Sigma}(\dR^n)$ and
$W^{k,\infty}_\Sigma(\dR^n\backslash\Sigma)$ of $L^\infty$-functions
which belong to $W^{k,\infty}$ in a neighbourhood or both one-sided
neighbourhoods of $\Sigma$, respectively.

% -------------------------------------------------------------------
\begin{dfn}
Let $\Sigma$, $\Omega_{\rm i}$, $\Omega_{\rm e}$, and the spaces
$W^{k,\infty}_{\Sigma}(\Omega_{\rm i})$ and $W^{k,\infty}_{\Sigma}(\Omega_{\rm e})$,
$k\in\dN_0$, be as above. Then we define
\begin{equation*}
\begin{split}
  W^{k,\infty}_{\Sigma}(\dR^n) \defequ \bigl\{ f\in L^\infty(\dR^n)\colon
  \exists~\text{domain}~\Omega'\subset\dR^n~\text{such that} \\
&\qquad\qquad\qquad\qquad\qquad\qquad\Omega'\supset\Sigma~\text{and}~f\upharpoonright\Omega'\in
  W^{k,\infty}(\Omega')\bigr\}, \\[0.5ex]
  W^{k,\infty}_\Sigma(\dR^n\backslash\Sigma) \defequ
  W^{k,\infty}_\Sigma(\Omega_{\rm i})\times W^{k,\infty}_\Sigma(\Omega_{\rm e}).
\end{split}
\end{equation*}
\end{dfn}

Finally, we recall a well-known result about the $\sS_{p,\infty}$ property
of bounded operators mapping into the Sobolev space $H^{q_2}(\Sigma)$,
where $q_2>0$ and $\Sigma=\partial\Omega_{\rm i}=\partial\Omega_{\rm e}$
is the $(n-1)$-dimensional compact connected
$C^\infty$-hypersurface from above, cf.\ \cite{G96} and \cite[Lemma~4.6]{BLL10}.

% -------------------------------------------------------------------
\begin{lem}\label{le.s_emb}
Let $\cK$ be a Hilbert space, $B\in\cB(\cK,L^2(\Sigma))$ and let $q_2
> q_1 \ge  0$. If $\ran B \subset H^{q_2}(\Sigma)$, then $B$ belongs to
the class $\sS_{\frac{n-1}{q_2 - q_1},\infty}(\cK, H^{q_1}(\Sigma))$.
\end{lem}

% *******************************************************************
% *******************************************************************
\section{Self-adjoint Schr\"odinger operators with
interactions on hypersurfaces}
\label{sec:3}
% *******************************************************************
% *******************************************************************

In this section we define the Schr\"odinger operators with $\delta$ and $\delta'$-interactions on hypersurfaces
with the help of quasi boundary triple techniques. These definitions coincide with the ones in the introduction
and are compatible with those for one-dimensional $\delta$-point interactions from \cite{AGHH05,AK99}
and the definition of $\delta$-interactions on manifolds via quadratic forms; see, e.g.\
\cite{BEKS94, EK03, EY02a, EY04, KV07}.
We also determine the semi-bounded closed quadratic form which corresponds to
the Schr\"odinger
operator with a $\delta'$-interaction on a hypersurface, which answers a question
from \cite{E08} posed by P.~Exner.
As a byproduct of the quasi boundary triple approach we obtain variants of
Krein's formula and the Birman--Schwinger principle.
This section contains the complete proofs of Theorem~A and Theorem~B from the introduction.

% *******************************************************************
\subsection{Notations and preliminary facts}
\label{sec:notations}
% *******************************************************************

Let $\Sigma$ be  a compact connected $C^\infty$-hypersurface
which separates the Euclidean space $\dR^n$, $n\geq 2$,
into a bounded (interior) domain $\Omega_{\rm i}$ and an unbounded
(exterior) domain $\Omega_{\rm e}$ with the common boundary
$\partial\Omega_{\rm i} = \partial\Omega_{\rm e} = \Sigma$.
Let
\begin{equation}
\label{Schrodinger}
  \cL = -\Delta + V,
\end{equation}
where $V$ is a real-valued potential from $L^\infty(\dR^n)$.
The restrictions of $\cL$ to the interior and exterior
domains will be denoted, respectively, by 
\[
\cL_{\rm i} = \cL\upharpoonright\Omega_{\rm i}\quad\text{and}\quad \cL_{\rm e}= \cL\upharpoonright\Omega_{\rm e}.
\]
For a function $f \in L^2(\dR^n)$ we write $f = f_{\rm i} \oplus f_{\rm e}$, where
$f_{\rm i} = f\upharpoonright\Omega_{\rm i}$ and
$f_{\rm e} = f\upharpoonright\Omega_{\rm e}$. Let us denote by $(\cdot,\cdot)$,
$(\cdot,\cdot)_{\rm i}$, $(\cdot,\cdot)_{\rm e}$ and $(\cdot,\cdot)_\Sigma$ the
inner products in the Hilbert spaces $L^2(\dR^n)$, $L^2(\Omega_{\rm i})$,
$L^2(\Omega_{\rm e})$ and $L^2(\Sigma)$, respectively. When it is clear
from the context, we denote the inner products in the Hilbert spaces $L^2(\dR^n;\dC^n)$,
$L^2(\Omega_{\rm i}; \dC^n)$, and $L^2(\Omega_{\rm e}; \dC^n)$
of vector-valued functions also by  $(\cdot,\cdot)$, $(\cdot,\cdot)_{\rm i}$ and
$(\cdot,\cdot)_{\rm e}$, respectively.

The minimal operators associated with the
differential expressions $\cL_{\rm i}$ and $\cL_{\rm e}$  are defined by
\begin{equation*}
\label{symop}
\begin{alignedat}{2}
  A_{\rm i} f_{\rm i} &= \cL_{\rm i} f_{\rm i},\qquad &
  \dom A_{\rm i} &= H^2_0(\Omega_{\rm i}),\\
  A_{\rm e} f_{\rm e} &= \cL_{\rm e} f_{\rm e},\qquad &
  \dom A_{\rm e} &= H^2_0(\Omega_{\rm e}).
\end{alignedat}
\end{equation*}
The operators $A_{\rm i}$ and $A_{\rm e}$ are densely defined closed symmetric
operators with infinite deficiency indices in
$L^2(\Omega_{\rm i})$ and $L^2(\Omega_{\rm e})$, respectively.
Hence their direct sum
\begin{equation}\label{aie}
  A_{\rm i,e} = A_{\rm i} \oplus A_{\rm e},\qquad \dom A_{\rm i,e}
  = H^2_0(\Omega_{\rm i})\oplus H^2_0(\Omega_{\rm e}),
\end{equation}
is a densely defined closed symmetric operator with infinite deficiency indices in the space
$L^2(\dR^n)=L^2(\Omega_{\rm i})\oplus L^2(\Omega_{\rm e})$.
Furthermore, we introduce the operators
\begin{equation*}
\label{opT}
\begin{alignedat}{2}
  T_{\rm i} f_{\rm i} &= \cL_{\rm i} f_{\rm i},
  \qquad & \dom T_{\rm i} &= H^{3/2}_\Delta(\Omega_{\rm i}),\\
  T_{\rm e} f_{\rm e} &= \cL_{\rm e} f_{\rm e},\qquad & \dom T_{\rm e} &= H^{3/2}_\Delta(\Omega_{\rm e}),
\end{alignedat}
\end{equation*}
and their direct sum
\[
  T_{\rm i,e} = T_{\rm i}\oplus T_{\rm e},\qquad
  \dom T_{\rm i,e} = H^{3/2}_\Delta(\dR^n\backslash\Sigma),
\]
where the notation in \eqref{opsops} is used. It can be shown that
$A_{\rm i}^* = \ov T_{\rm i}$, $A_{\rm e}^* = \ov T_{\rm e}$,
and hence $A_{\rm i,e}^* = \ov T_{\rm i,e}$.
Next we define the usual self-adjoint Dirichlet and Neumann realizations of the differential
expressions $\cL_{\rm i}$ and $\cL_{\rm e}$ in $L^2(\Omega_{\rm i})$
and $L^2(\Omega_{\rm e})$, respectively:
\begin{equation*}
\label{ADAN}
\begin{alignedat}{2}
  A_{\rm D, i}f_{\rm i} &= \cL_{\rm i} f_{\rm i},\qquad &
    \dom A_{\rm D, i} &= \bigl\{f_{\rm i}\in H^2(\Omega_{\rm i})\colon f_{\rm i}|_\Sigma= 0 \bigr\},\\
  A_{\rm D, e}f_{\rm e} &= \cL_{\rm e} f_{\rm e},\qquad &
    \dom A_{\rm D, e} &= \bigl\{f_{\rm e}\in H^2(\Omega_{\rm e})\colon
    f_{\rm e}|_\Sigma= 0 \bigr\},\\
  A_{\rm N, i}f_{\rm i} &= \cL_{\rm i} f_{\rm i},\qquad &
    \dom A_{\rm N, i} &= \bigl\{f_{\rm i}\in H^2(\Omega_{\rm i})\colon
    \partial_{\nu_{\rm i}}f_{\rm i}|_\Sigma = 0 \bigr\},\\
  A_{\rm N, e}f_{\rm e} &= \cL_{\rm e} f_{\rm e},\qquad &
    \dom A_{\rm N, e} &= \bigl\{f_{\rm e}\in H^2(\Omega_{\rm e})\colon
     \partial_{\nu_{\rm e}}f_{\rm e}|_\Sigma = 0 \bigr\},
\end{alignedat}
\end{equation*}
and their direct sums
\begin{equation}
\label{ADANoplus1}
\begin{alignedat}{2}
 A_{\rm D, i,e } &= A_{\rm D,i}\oplus A_{\rm D,e},\\
 \dom A_{\rm D, i,e} &= \bigl\{f\in H^2(\dR^n\backslash\Sigma):
  f_{\rm i}|_\Sigma = f_{\rm e}|_\Sigma = 0 \bigr\},
\end{alignedat}
\end{equation}
and
\begin{equation}
\label{ADANoplus2}
\begin{alignedat}{2}
   A_{\rm N, i,e } &= A_{\rm N,i}\oplus A_{\rm N,e},\\
    \dom A_{\rm N, i,e} &= \bigl\{f\in H^2(\dR^n\backslash\Sigma):
    \partial_{\nu_{\rm i}}f_{\rm i}|_\Sigma =
 \partial_{\nu_{\rm e}}f_{\rm e}|_\Sigma = 0 \bigr\},
\end{alignedat}
\end{equation}
which are self-adjoint operators in $L^2(\dR^n)$.
Finally, we denote the usual self-adjoint (free) realization of $\cL$ in $L^2(\dR^n)$ by
\begin{equation}
\label{Afree}
  A_{\rm free} f = \cL f,\qquad \dom A_{\rm free} = H^2(\dR^n).
\end{equation}

In the next proposition we define quasi boundary triples
for $A_{\rm i}^*$ and $A_{\rm e}^*$, and recall some properties of
the associated $\gamma$-fields and Weyl functions; see \cite[Proposition~4.6]{BL07}
and \cite[Theorem~4.2]{BLL10}. For brevity we discuss the interior case $j={\rm i}$
and the exterior case $j={\rm e}$ simultaneously.

% -------------------------------------------------------------------
\begin{prop}
\label{prop:qbt0}
Let $A_{\rm i}$,  $A_{\rm e}$, $T_{\rm i}$, $T_{\rm e}$, $A_{\rm D, i}$,
$A_{\rm D, e}$, $A_{\rm N,i}$ and $A_{\rm N,e}$ be as above.
Then the following statements hold for $j={\rm i}$ and $j={\rm e}$.
\begin{itemize}\setlength{\itemsep}{1.2ex}
\item[\rm (i)]
The triple
$\Pi_j = \{L^2(\Sigma), \Gamma_{0,j},\Gamma_{1,j}\}$, where
\begin{equation*}
  \Gamma_{0,j} f_j = \partial_{\nu_j} f_j|_\Sigma, \quad
  \Gamma_{1,j} f_j = f_j|_\Sigma, \quad f_j\in \dom T_j=H^{3/2}_\Delta(\Omega_j),
\end{equation*}
is a quasi boundary triple for $A^*_j$.
The restrictions of $T_j$ to the kernels of the boundary
mappings are the Neumann and Dirichlet operators:
\begin{equation*}
  T_j \upharpoonright \ker \Gamma_{0,j} = A_{{\rm N},j}, \quad
  T_j \upharpoonright \ker \Gamma_{1,j} = A_{{\rm D},j};
\end{equation*}
the ranges of the boundary mappings are
\[
  \ran \Gamma_{0,j} = L^2(\Sigma)
  \quad\text{and}\quad
  \ran \Gamma_{1,j} = H^1(\Sigma).
\]

\item[\rm (ii)]
For $\lambda\in\rho(A_{{\rm N}, j})$ and $\varphi \in L^2(\Sigma)$ the
boundary value problem
\begin{equation}
\label{problemie}
\begin{split}
(\cL_j -\lambda)f_{ j} = 0, \qquad  \partial_{\nu_j}f_{j}|_\Sigma = \varphi,
\end{split}
\end{equation}
has the unique solution $\gamma_j(\lambda)\varphi\in H^{3/2}_\Delta(\Omega_j)$, where $\gamma_j$
is the $\gamma$-field associated with $\Pi_j$. Moreover, $\gamma_j(\lambda)$ is bounded from  $L^2(\Sigma)$
into $L^2(\Omega_j)$.

\item[\rm (iii)]
For $\lambda\in\rho(A_{{\rm N},j})$ the Weyl function $M_j$ associated with $\Pi_j$ is given by\[
  M_j(\lambda)\varphi =f_{j}|_\Sigma, \qquad
  \varphi \in L^2(\Sigma),
\]
where $f_{j}=\gamma_j(\lambda)\varphi$ is the solution of \eqref{problemie}. The operators $M_j(\lambda)$
are bounded
from $L^2(\Sigma)$ to $H^1(\Sigma)$ and compact in $L^2(\Sigma)$.
If, in addition, $\lambda\in\rho(A_{{\rm D},j})$, then $M_j(\lambda)$ is a bijective
map from $L^2(\Sigma)$ onto $H^1(\Sigma)$.
\end{itemize}
\end{prop}

\noindent
The operators $M_{\rm i}(\lambda)$ and $M_{\rm e}(\lambda)$ in Proposition~\ref{prop:qbt0}\,(iii) are
the Neumann-to-Dirichlet maps associated with the
differential expressions $\cL_{\rm i}-\lambda$ and $\cL_{\rm e} -\lambda$, respectively.

% *******************************************************************
\subsection{Schr\"odinger operators with $\delta$-interactions on
hypersurfaces: self-adjoint\-ness,  Krein's formula and $H^2$-regularity}
\label{sec:delta}
% *******************************************************************

In this section we make use of quasi boundary triples to define and study the Schr\"odinger operator
$A_{\delta,\alpha}$ associated with the formal differential expression
$\cL_{\delta,\alpha}=-\Delta+V - \alpha\langle\delta_\Sigma,\cdot\,\rangle\,\delta_\Sigma$ in \eqref{expressions1}.
It is convenient to use the symmetric extension
\begin{equation}
\label{wtA}
  \wt A \defeq   A_{\rm free} \cap A_{\rm D, i, e}=
  \cL\upharpoonright\big\{f\in H^2(\dR^n)\colon f_{\rm i}|_{\Sigma} = f_{\rm e}|_{\Sigma} = 0\big\}
\end{equation}
of the orthogonal sum $A_{\rm i,e}$ in \eqref{aie} as the underlying symmetric operator
for the quasi boundary triple. Furthermore,
\begin{equation}
\label{wtT}
  \wt T  \defeq T_{\rm i,e} \upharpoonright
  \bigl\{f_{\rm i} \oplus f_{\rm e} \in  H^{3/2}_\Delta(\dR^n\backslash\Sigma)\colon
  f_{\rm i}|_\Sigma = f_{\rm e}|_\Sigma\bigr\}.
\end{equation}
acts as the operator on whose domain boundary mappings are defined in the next proposition.
The method of intermediate extensions is inspired by the general considerations for
ordinary boundary triples in \cite[Section~4]{DHMS00}. We remark that the quasi
boundary triple and Weyl function below appear also implicitly in \cite{AP04}
and \cite[Section~4]{R09} in a different context.

\begin{prop}
\label{prop:qbt1}
Let the operators $\wt A$, $\wt T$, $A_{\rm D, i,e}$, $A_{\rm N,i,e}$ and $A_{\rm free}$
be as in~\eqref{wtA}, \eqref{wtT}, \eqref{ADANoplus1}, \eqref{ADANoplus2} and \eqref{Afree}, respectively, and
let $M_{\rm i}$ and $M_{\rm e}$ be the Weyl functions from Proposition~\ref{prop:qbt0}.
Then the following statements hold.
\begin{itemize}\setlength{\itemsep}{1.2ex}
\item[\rm (i)]
The triple $\wt\Pi = \{L^2(\Sigma), \wt\Gamma_0,\wt\Gamma_1\}$, where
\begin{equation*}
  \wt\Gamma_0 f =
  \partial_{\rm \nu_{\rm e}} f_{\rm e}|_\Sigma+ \partial_{\rm \nu_{\rm i}} f_{\rm i}|_\Sigma,
  \quad
  \wt\Gamma_1f = f|_\Sigma,
  \qquad
  f=f_{\rm i}\oplus f_{\rm e} \in\dom\wt T,
\end{equation*}
is a quasi boundary triple for $\wt A^*$.
The restrictions of $\wt T$ to the kernels of the boundary
mappings are
\[
  \wt T \upharpoonright \ker\wt \Gamma_0 = A_{\rm free}
  \quad\text{and}\quad
  \wt T \upharpoonright \ker\wt \Gamma_1 = A_{\rm D,i,e},
\]
and the ranges of the boundary mappings are
\begin{equation}
\label{ran_wtG}
  \ran \wt \Gamma_0 = L^2(\Sigma)
  \quad\text{and}\quad
  \ran \wt \Gamma_1 = H^1(\Sigma).
\end{equation}

\item[\rm (ii)]
For $\lambda\in\rho(A_{\rm free})$ and $\varphi\in L^2(\Sigma)$ the
transmission problem
\begin{equation}
\label{transm_i32}
  (\cL-\lambda) f = 0,
  \qquad  f_{\rm e}|_\Sigma =  f_{\rm i}|_\Sigma, \quad \partial_{\nu_{\rm e}} f_{\rm e}|_\Sigma
  + \partial_{\nu_{\rm i}} f_{\rm i}|_\Sigma = \varphi,
\end{equation}
has the unique solution $\wt\gamma(\lambda)\varphi\in H^{3/2}_\Delta(\dR^n\backslash\Sigma)$,
where $\wt\gamma$ is the $\gamma$-field
associated with $\wt\Pi$.  Moreover, $\wt\gamma(\lambda)$ is bounded from
$L^2(\Sigma)$ to $L^2(\dR^n)$.

\item[\rm (iii)]
For $\lambda\in\rho(A_{\rm free})$ the values $\wt M(\lambda)$ of the Weyl function
associated with $\wt\Pi$ are bounded operators from $L^2(\Sigma)$ to $H^1(\Sigma)$
and compact operators in $L^2(\Sigma)$.
If, in addition, $\lambda\in\rho(A_{\rm D,i,e})$, then $\wt M(\lambda)$ is a bijective map from
$L^2(\Sigma)$ onto $H^1(\Sigma)$.
Moreover,
the identity
\begin{equation}
\label{weyl1}
  \wt M(\lambda) = \bigl(M_{\rm i}(\lambda)^{-1}+M_{\rm e}(\lambda)^{-1}\bigr)^{-1}
\end{equation}
holds for all
$\lambda\in\rho(A_{\rm free})\cap\rho(A_{\rm D,i,e})\cap\rho(A_{\rm N,i,e})$.
\end{itemize}
\end{prop}

\begin{proof}
(i)
First note that the boundary mappings $\wt\Gamma_0$, $\wt\Gamma_1$ are well defined
because of the properties of the trace mappings \eqref{LMtrace}.
We show that the triple $\wt\Pi$ satisfies the conditions (a), (b) and (c)
in Proposition~\ref{suff_cond_qbt}.
For condition (a), let $\varphi\in H^{1/2}(\Sigma)$ and $\psi \in H^{3/2}(\Sigma)$
be arbitrary.
By~\eqref{Sobtrace} there exist $f_{\rm i}\in H^2(\Omega_{\rm i})$ and
$f_{\rm e}\in H^2(\Omega_{\rm e})$ such that
\[
  \partial_{\nu_{\rm i}}f_{\rm i}|_\Sigma = \varphi, \qquad
  f_{\rm i}|_\Sigma = \psi, \qquad
  \partial_{\nu_{\rm e}}f_{\rm e}|_\Sigma = 0, \qquad
  f_{\rm e}|_\Sigma = \psi.
\]
Since $H^2(\dR^n\backslash\Sigma) \subset H^{3/2}_\Delta(\dR^n\backslash\Sigma)$,
we have $f \defeq f_{\rm i}\oplus f_{\rm e}\in\dom\wt T$ and $\wt\Gamma_0f = \varphi$,
$\wt \Gamma_1f = \psi$. Hence
\[
   H^{1/2}(\Sigma)\times H^{3/2}(\Sigma)\subset
  \ran \biggl(\,\begin{matrix} \wt\Gamma_0 \\ \wt\Gamma_1 \end{matrix}\biggr),
\]
which implies that the first item in (a) of Proposition~\ref{suff_cond_qbt} is satisfied; the second item is clear.
Next let $f = f_{\rm i} \oplus f_{\rm e}$ and $g = g_{\rm i} \oplus g_{\rm e}$ be
two arbitrary functions in $\dom \wt T$. From Green's identity~\eqref{green_identity}
we obtain the following two equalities:
\begin{equation*}
\label{eq:green1}
\begin{split}
  (T_{\rm i} f_{\rm i},g_{\rm i})_{\rm i} - ( f_{\rm i},T_{\rm i}g_{\rm i})_{\rm i}
  &= \bigl(f_{\rm i}|_\Sigma, \partial_{\nu_{\rm i}} g_{\rm i}|_\Sigma\bigr)_\Sigma
  -\bigl(\partial_{\nu_{\rm i}} f_{\rm i}|_\Sigma, g_{\rm i}|_\Sigma\bigr)_\Sigma, \\[0.5ex]
  (T_{\rm e} f_{\rm e},g_{\rm e})_{\rm e} - ( f_{\rm e},T_{\rm e}g_{\rm e})_{\rm e}
  &= \bigl(f_{\rm e}|_\Sigma, \partial_{\nu_{\rm e}} g_{\rm e}|_\Sigma\bigr)_\Sigma -
  \bigl(\partial_{\nu_{\rm e}} f_{\rm e}|_\Sigma, g_{\rm e}|_\Sigma\bigr)_\Sigma.
\end{split}
\end{equation*}
Since the functions  $f$ and $g$ in $\dom \wt T$ satisfy the boundary conditions
$f_{\rm i}|_\Sigma = f_{\rm e}|_\Sigma = f|_\Sigma$ and
$g_{\rm i}|_\Sigma = g_{\rm e}|_\Sigma=g|_\Sigma$, we have
\begin{equation}
\label{green1}
\begin{split}
  \bigl(\wt T f,g\bigr) - \bigl(f,\wt T g\bigr)
  &= (T_{\rm i} f_{\rm i},g_{\rm i})_{\rm i}
  -(f_{\rm i},T_{\rm i}g_{\rm i})_{\rm i}
  +(T_{\rm e} f_{\rm e},g_{\rm e})_{\rm e}
  -(f_{\rm e},T_{\rm e}g_{\rm e})_{\rm e} \\[0.5ex]
  &= \bigl(f|_\Sigma,\partial_{\nu_{\rm i}}g_{\rm i}|_\Sigma
  +\partial_{\nu_{\rm e}}g_{\rm e}|_\Sigma\bigr)_\Sigma
  -\bigl(\partial_{\nu_{\rm i}}f_{\rm i}|_\Sigma+\partial_{\nu_{\rm e}}f_{\rm e}|_\Sigma,
  g|_\Sigma\bigr)_\Sigma,
\end{split}
\end{equation}
which shows that condition (b) of Proposition~\ref{suff_cond_qbt} is fulfilled.
Since the restriction $\wt T\upharpoonright\ker \wt\Gamma_0$ contains the self-adjoint
operator $A_{\rm free}$, also condition (c) is satisfied.
Hence we can apply Proposition~\ref{suff_cond_qbt}, which implies that
$\wt T\upharpoonright(\ker\wt \Gamma_0\cap \ker\wt \Gamma_1)$ is a
densely defined closed symmetric operator, that the triple
$\wt \Pi = \{L^2(\Sigma),\wt \Gamma_0,\wt\Gamma_1\}$ is a quasi boundary triple
for its adjoint and that $A_{\rm free} = \wt T\upharpoonright\ker \wt\Gamma_0$.
Note that the operator $\wt T\upharpoonright\ker\wt \Gamma_1$ is symmetric by \eqref{green1}
and it contains the self-adjoint operator $A_{\rm D,i,e}$.
Therefore these operators also coincide.
Hence we get
\[
\wt T\upharpoonright(\ker\wt \Gamma_0\cap \ker\wt \Gamma_1)=
\big(\wt T\upharpoonright\ker\wt \Gamma_0\big)\cap \big(\wt T\upharpoonright \ker\wt \Gamma_1\big)
= A_{\rm free}\cap A_{\rm D,i,e} = \wt A.
\]

Since, for $j={\rm i}$ and $j={\rm e}$, the mapping $f_j\mapsto f_j|_\Sigma$
is surjective from $H^{3/2}_\Delta(\Omega_j)$ onto $H^1(\Sigma)$ and
the mapping $f_j\mapsto \partial_{\nu_{\rm j}}f_j|_\Sigma$ is surjective
from $H^{3/2}_\Delta(\Omega_j)$ onto $L^2(\Sigma)$, it follows
easily that $\ran\wt\Gamma_1=H^1(\Sigma)$ and that $\ran \wt\Gamma_0 \subset L^2(\Sigma)$.
In order to see that $\wt \Gamma_0$ maps surjectively onto $L^2(\Sigma)$,
let us fix an arbitrary $\chi\in C^\infty_0(\dR^n)$ such that $\chi\equiv1$ on
an open neighbourhood of $\overline{\Omega_{\rm i}}$.
Let ${\rm SL}$ be the single-layer potential associated with the hypersurface $\Sigma$ and
the differential expression $- \Delta + 1$;
see, e.g.\ \cite[Chapter~6]{McL00} for the definition and properties of single-layer potentials.
By \cite[Theorem~6.11, Theorem~6.12\,(i)]{McL00}, for an arbitrary $\varphi\in L^2(\Sigma)$,
the function $f \defeq \chi {\rm  SL}\varphi$ belongs to $\dom \wt T$ and satisfies the condition
\[
  \partial_{\nu_{\rm e}}f_{\rm e}|_\Sigma + \partial_{\nu_{\rm i}}f_{\rm i}|_\Sigma = \varphi,
\]
hence $\wt \Gamma_0f = \varphi$, and thus $\ran\wt \Gamma_0 = L^2(\Sigma)$.

(ii) For $\lambda\in\rho(A_{\rm free})$ the $\gamma$-field $\widetilde\gamma(\lambda)$
associated with the quasi boundary triple $\wt\Pi$ maps $\ran\wt\Gamma_0=L^2(\Sigma)$
onto $\ker(\wt T-\lambda)$ by Definition~\ref{def:gammaWeyl}
and Proposition~\ref{gammaprop}\,(i). Hence $f=f_{\rm i}\oplus f_{\rm e}\defeq \wt\gamma(\lambda)\varphi$
satisfies $(\cL-\lambda)f=0$, $f\in H^{3/2}(\dR^n\backslash\Sigma)$
and $f_{\rm i}\vert_\Sigma = f_{\rm e}\vert_\Sigma$. Furthermore,
\[
  \varphi = \wt\Gamma_0 \wt\gamma(\lambda)\varphi
  = \wt\Gamma_0 f
  = \partial_{\rm \nu_{\rm e}} f_{\rm e}|_\Sigma+ \partial_{\rm \nu_{\rm i}} f_{\rm i}|_\Sigma
\]
and hence $f=\wt\gamma(\lambda)\varphi$ is the unique solution of the problem \eqref{transm_i32}.

(iii)
Definition~\ref{def:gammaWeyl}, Proposition~\ref{gammaprop}\,(iv) and\,(v)
and \eqref{ran_wtG} imply that $\wt M(\lambda)$ is a bounded operator from
$L^2(\Sigma)$ into $H^1(\Sigma)$ for $\lambda\in\rho(A_{\rm free})$ and
that it is bijective for $\lambda\in\rho(A_{\rm free})\cap\rho(A_{\rm D,i,e})$.
The compactness of $\wt M(\lambda)$ in $L^2(\Sigma)$ is a consequence of
the compactness of the embedding of $H^1(\Sigma)$ into $L^2(\Sigma)$;
see, e.g.\ \cite[Theorem~7.10]{W87}.

In order to prove the identity \eqref{weyl1}, let
$\lambda\in\rho(A_{\rm free})\cap\rho(A_{\rm D,i,e})\cap\rho(A_{\rm N,i,e})$.
For such $\lambda$ the operator $\wt M(\lambda)$ is invertible, and the same holds true
for $M_{\rm i}(\lambda)$ and $M_{\rm e}(\lambda)$; cf.\ Proposition~\ref{prop:qbt0}.
If $\wt M(\lambda)\varphi=\psi$ for some $\varphi\in L^2(\Sigma)$ and $\psi\in H^1(\Sigma)$,
then there exists an $f=f_{\rm i}\oplus f_{\rm e}\in\ker(\wt T-\lambda)$ such that
\[
  \wt\Gamma_0 f = \varphi \qquad\text{and}\qquad \wt\Gamma_1 f = \psi.
\]
As $f_{\rm i}\in\ker(T_{\rm i}-\lambda)$ and $f_{\rm e}\in\ker(T_{\rm e}-\lambda)$,
we have
\begin{align*}
  \Gamma_{0,\rm i}f_{\rm i} &= M_{\rm i}(\lambda)^{-1}\Gamma_{1,\rm i}f_{\rm i}
  = M_{\rm i}(\lambda)^{-1}\psi, \\[0.5ex]
  \Gamma_{0,\rm e}f_{\rm e} &= M_{\rm e}(\lambda)^{-1}\Gamma_{1,\rm e}f_{\rm e}
  = M_{\rm e}(\lambda)^{-1}\psi,
\end{align*}
and hence
\begin{align*}
  \wt M(\lambda)^{-1}\psi &= \varphi
  = \partial_{\nu_{\rm i}}f_{\rm i}|_\Sigma + \partial_{\nu_{\rm e}}f_{\rm e}|_\Sigma
  = \Gamma_{0,\rm i}f_{\rm i} + \Gamma_{0,\rm e}f_{\rm e} \\[0.5ex]
  &= M_{\rm i}(\lambda)^{-1}\psi + M_{\rm e}(\lambda)^{-1}\psi.
\end{align*}
Since this is true for arbitrary $\psi\in H^1(\Sigma)$, relation \eqref{weyl1} follows.
\end{proof}

\begin{rem}
Assume for simplicity that the potential $V$ in the differential expression $\cL$
in \eqref{Schrodinger} is identically equal to zero.  In this case the
$\gamma$-field $\wt \gamma$ and the Weyl function $\wt M$
in Proposition~\ref{prop:qbt1} are, roughly speaking, extensions of the acoustic
single-layer potential for the Helmholtz equation. \label{acoustic}
In fact, if $G_\lambda$, $\lambda\in\dC\setminus\dR$, is the integral kernel
of the resolvent of $A_{\rm free}$,
then for all
$\varphi\in C^\infty(\Sigma)$ we have
\[
  \big(\wt\gamma(\lambda)\varphi\big)(x)
  = \int_\Sigma G_\lambda(x,y)\varphi(y)d\sigma_y, \quad x\in\dR^n\backslash\Sigma,
\]
and
\[
  \big(\wt M(\lambda)\varphi\big)(x)
  = \int_\Sigma G_\lambda(x,y)\varphi(y)d\sigma_y,\quad x\in\Sigma,
\]
where $\sigma_y$ is the natural Lebesgue measure on $\Sigma$.
For more details we refer the reader to~\cite[Chapter~6]{McL00}; see also \cite{CK83, Co88}.
\end{rem}
%---------------------------------------------------------------------------

We repeat the definition of a Schr\"odinger operator with $\delta$-potential
from the introduction and relate it to the quasi boundary triple $\wt\Pi$.
\begin{dfn}
\label{def:delta}
For a real-valued function $\alpha \in L^\infty(\Sigma)$ the Schr\"odinger
operator with $\delta$-potential on the hypersurface $\Sigma$ and strength
$\alpha$ is defined as follows:
\[
  A_{\delta,\alpha} \defeq \wt T \upharpoonright \ker(\alpha\wt \Gamma_1-\wt\Gamma_0),
\]
which is equivalent to
\begin{equation}
\label{eq:delta}
\begin{split}
  A_{\delta,\alpha}f \defequ -\Delta f+Vf, \\
  \dom A_{\delta,\alpha} \defequ \Biggl\{f\in H^{3/2}_\Delta(\dR^n\backslash\Sigma):\begin{matrix}
  f_{\rm i}\vert_\Sigma=f_{\rm e}|_\Sigma=f|_\Sigma\\ 
\alpha f|_\Sigma=
  \partial_{\nu_{\rm e}}f_{\rm e}|_\Sigma+\partial_{\nu_{\rm i}}f_{\rm i}|_\Sigma\end{matrix} \Biggr\}.
\end{split}
\end{equation}
\end{dfn}

The definition of $A_{\delta,\alpha}$ is compatible with the definition of a
point $\delta$-interaction in the one-dimensional case~\cite[Section~I.3]{AGHH05}, \cite{AK99} and
the definitions of the operators with $\delta$-potentials on hypersurfaces given in~\cite{AGS87, S88}
and in~\cite{BEKS94}; see also Proposition~\ref{prop:exner}. Note also that the domain of $A_{\delta,\alpha}$ is contained in $H^1(\dR^n)$;
cf.\ Proposition~\ref{prop:exner}.

%---Fig. 1----------------------------------------------------------------------------------
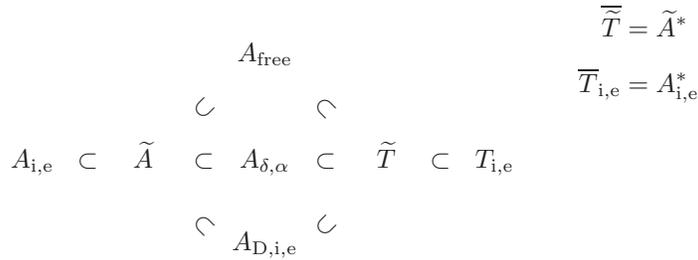
\begin{figure}[H]
\begin{displaymath}
   \xymatrix{
& & \text{$A_{\rm free}$}\ar@{}[dr] | {\begin{turn}{-45}$\subset$\end{turn}}\\
%----------------------------------------------------
\text{$A_{\rm i,e}$} \ar@{}[r] |{\begin{turn}{0}$\subset$\end{turn}} &
\text{$\wt A$}
\ar@{}[r] |{\begin{turn}{0}$\subset$\end{turn}}
\ar@{}[ur] | {\begin{turn}{45}$\subset$\end{turn}}
\ar@{}[dr] | {\begin{turn}{-45}$\subset$\end{turn}} &
\text{$A_{\delta,\alpha}$} \ar@{}[r] |{\begin{turn}{0}$\subset$\end{turn}} &
\text{$\wt T$} \ar@{}[r] |{\begin{turn}{0}$\subset$\end{turn}}&
\text{$T_{\rm i,e}$} \\
%-----------------------------------------------------------------
&& \text{\raisebox{4ex}{$A_{\rm D,i,e}$}}\ar@{}[ur] | {\begin{turn}{45}$\subset$\end{turn}}}
\qquad \begin{aligned} \ov{\wt T} &= \wt A^* \\[1ex] \ov{T}_{\rm i,e} &= A_{\rm i,e}^* \end{aligned}
\end{displaymath}
\caption{This figure shows how the operator $A_{\delta,\alpha}$ is related to the other operators
studied in this section. The operators $A_{\rm free}$, $A_{\delta,\alpha}$ and $A_{\rm D,i,e}$
are self-adjoint.}
\label{fig1}
\end{figure}

%-------------------------------------------------------------------------------------

The next theorem contains a proof of self-adjointness of $A_{\delta,\alpha}$
and provides a factorization for the resolvent difference of $A_{\delta,\alpha}$
and $A_{\rm free}$ via Krein's formula; cf.\ \cite[Lemma 2.3\,(iii)]{BEKS94}.
Item (iii) in Theorem~\ref{thm:delta} can be viewed as a variant of the Birman--Schwinger principle;
it coincides with the one in \cite{BEKS94}.
The first item of Theorem~\ref{thm:delta} is part of Theorem A in the introduction.

%-------------------------------------------------------------------------------------
\begin{thm}
\label{thm:delta}
Let $A_{\delta,\alpha}$ be as above and let $A_{\rm free}$ be the self-adjoint
operator defined in~\eqref{Afree}.
Let $\wt \gamma$ and $\wt M$ be the $\gamma$-field and the Weyl function associated
with the quasi boundary triple $\wt \Pi$ from Proposition~\ref{prop:qbt1}.
Then the following statements hold.
\begin{itemize}\setlength{\itemsep}{1.2ex}
%--Self-adjointness---------------------------------------------------
\item[\rm (i)] The operator $A_{\delta,\alpha}$ is self-adjoint in the Hilbert space $L^2(\dR^n)$.
%--Krein's formula---------------------------------------------------------
\item[\rm (ii)] For all $\lambda\in\rho(A_{\delta,\alpha})\cap\rho(A_{\rm free})$ the following
Krein formula holds:
\[
  (A_{\delta,\alpha} - \lambda)^{-1} -(A_{\rm free} - \lambda)^{-1}
  = \wt\gamma(\lambda)\,\bigl(I - \alpha\wt M(\lambda)\bigr)^{-1}\alpha\, \wt\gamma(\ov\lambda)^*,
\]
where $(I - \alpha \wt M(\lambda))^{-1}\in\cB(L^2(\Sigma))$.
\item[\rm (iii)] For all $\lambda\in\dR\setminus\sigma(A_{\rm free})$ we have
\[\lambda\in\sigma_p(A_{\delta,\alpha})\quad\Longleftrightarrow \quad
0\in\sigma_p\bigl(I - \alpha\wt M(\lambda)\bigr)\]
and $\dim\ker(A_{\delta,\alpha}-\lambda)=\dim\ker (I - \alpha\wt M(\lambda))$.
\end{itemize}
\end{thm}
%-------------------------------------------------------------------------------------

\begin{proof}
Under our assumptions on the function $\alpha$ the operator of multiplication
with $\alpha$ is bounded and self-adjoint in the Hilbert space $L^2(\Sigma)$.
The values of the Weyl function $\wt M$ are compact operators in $L^2(\Sigma)$;
see Proposition~\ref{prop:qbt1}\,(iii). Now the assertions (i)--(iii) follow from
Theorem~\ref{thm:krein}.
\end{proof}

The next theorem gives assumptions on $\alpha$, which ensure that the domain
of the self-adjoint operator $A_{\delta,\alpha}$ has $H^2$-regularity in $\dR^n\backslash\Sigma$. This theorem is the first
part of Theorem B in the introduction. Recall that $W^{1,\infty}(\Sigma)$ is the Soboloev space of order one of $L^\infty$
functions on $\Sigma$; cf. Section~\ref{sec:fspaces}.

\begin{thm}
\label{thm:H2delta}
Let $A_{\delta,\alpha}$ be the self-adjoint Schr\"odinger operator in Definition~\ref{def:delta}
and assume, in addition,
that the function $\alpha\colon\Sigma\rightarrow\dR$ belongs to $W^{1,\infty}(\Sigma)$.
Then $\dom A_{\delta,\alpha}$ is contained in $H^2(\dR^n\backslash\Sigma)$.
\end{thm}
%-------------------------------------------------------------------------------------

\begin{proof}
For any function $f\in \dom A_{\delta,\alpha}$ we have $f\in\dom \wt T
\subset H^{3/2}_\Delta(\dR^n\backslash\Sigma)$. Then by Proposition~\ref{prop:qbt1}\,(i)
\[
  \wt\Gamma_1f \in H^1(\Sigma).
\]
The definition of the operator $A_{\delta,\alpha}$, the assumptions
on the smoothness of $\alpha$ and the property \eqref{W1Hk} imply that
\begin{equation}
\label{reg1}
  \wt\Gamma_0f = \alpha\wt\Gamma_1f \in H^1(\Sigma).
\end{equation}
Let us fix $\lambda\in\dC\setminus\dR$.
By the standard decomposition
\begin{equation}
\label{decomposition2}
  \dom \wt T = \dom A_{\rm free} \dotplus \ker(\wt T-\lambda)
\end{equation}
the function $f \in \dom A_{\delta,\alpha}$ can be represented
in the form $f= f_{\rm free} + f_\lambda$, where $f_{\rm free} \in \dom A_{\rm free}$
and $f_\lambda\in\ker(\wt T-\lambda)$.
It is clear that 
\[
f_{\rm free} \in H^2(\dR^n)\subset H^2(\dR^n\backslash\Sigma).
\]
Relation~\eqref{reg1} and $A_{\rm free}=\widetilde T\upharpoonright \ker \widetilde\Gamma_0$ yield
\begin{equation}
  \label{reg2}
  \wt\Gamma_0f_\lambda = \wt\Gamma_0f \in H^1(\Sigma) \subset H^{1/2}(\Sigma).
\end{equation}
The properties of the trace map in \eqref{Sobtrace} show that
$\wt \Gamma_0$ maps the space $\dom \wt T \cap H^2(\dR^n\backslash\Sigma)$
onto $H^{1/2}(\Sigma)$, and hence \eqref{decomposition2} implies that $\wt\Gamma_0$ maps
\[
  \ker(\wt T-\lambda) \cap H^2(\dR^n\backslash\Sigma)
\]
bijectively onto $H^{1/2}(\Sigma)$.
This observation and~\eqref{reg2} show  $f_\lambda\in H^2(\dR^n\backslash\Sigma)$,
and therefore $f=f_{\rm free} + f_\lambda\in H^2(\dR^n\backslash\Sigma)$.
\end{proof}

It follows from the proof that for Theorem~\ref{thm:H2delta} to hold it is sufficient that the
multiplication by $\alpha$ maps $H^1(\Sigma)$-functions into $H^{1/2}(\Sigma)$.

\medskip

A common method to define self-adjoint Schr\"odinger operators with $\delta$-interactions
on hypersurfaces makes use of semi-bounded closed sesquilinear forms.
For this consider the sesquilinear form
\begin{equation}
\label{sform}
  \fra_{\delta,\alpha}[f, g] = \bigl(\nabla f,\nabla g\bigr) + \bigl(Vf, g \bigr)
  - \bigl(\alpha f|_\Sigma, g|_\Sigma\bigr)_\Sigma,\qquad f,g\in H^1(\dR^n).
\end{equation}
As it is shown in~\cite{BEKS94}, for a real-valued $\alpha\in L^\infty(\Sigma)$
and a real-valued $V\in L^\infty(\dR^n)$, the form $\fra_{\delta,\alpha}$
is semi-bounded, closed and symmetric. The first representation
theorem --- see \cite[Theorem~VI.2.1]{K95} or \cite[Theorem VIII.15]{RS72-I} ---
yields that a unique self-adjoint operator $\cA_{\delta,\alpha}$ in $L^2(\dR^n)$
corresponds to the form $\fra_{\delta,\alpha}$
in the sense that
\[
  (\cA_{\delta,\alpha}f,g)=\fra_{\delta,\alpha}[f,g] \qquad\text{for all }
  f\in\dom\cA_{\delta,\alpha}\text{ and }g\in\dom\fra_{\delta,\alpha}=H^1(\RR^n).
\]
In the next proposition we show that our approach leads to the same operator.

%-------------------------------------------------------------------------------------
\begin{prop}
\label{prop:exner}
The self-adjoint Schr\"odinger operator $A_{\delta,\alpha}$ in Definition~\ref{def:delta} and the
self-adjoint operator $\cA_{\delta,\alpha}$ corresponding to the sesquilinear form in~\eqref{sform} coincide.
\end{prop}
%-------------------------------------------------------------------------------------

\begin{proof}
First we show the inclusion $\dom A_{\delta,\alpha} \subset \dom \fra_{\delta,\alpha}$.
For this let $f = f_{\rm i}\oplus f_{\rm e}\in \dom A_{\delta,\alpha}$.
According to \eqref{eq:delta} we have, in particular,
\[
  f_{\rm i} \in H^{3/2}(\Omega_{\rm i})\subset H^1(\Omega_{\rm i}),\quad
  f_{\rm e} \in H^{3/2}(\Omega_{\rm e})\subset H^1(\Omega_{\rm e}),
  \quad\text{and}\quad f_{\rm i}|_\Sigma = f_{\rm e}|_\Sigma.
\]
Making use of \cite[Theorems~5.24 and 5.29]{AF03} a standard extension argument
implies that $f\in H^1(\dR^n)$ and hence $\dom A_{\delta,\alpha} \subset \dom \fra_{\delta,\alpha}$.

Next let $f = f_{\rm i}\oplus f_{\rm e}\in\dom A_{\delta,\alpha}$
and $g = g_{\rm i}\oplus g_{\rm e}\in\dom \fra_{\delta,\alpha}$. Then $\fra_{\delta,\alpha}[f,g]$
is well defined.
By the first Green's identity \eqref{half_green} we have
\begin{align*}
  (\nabla f_{\rm i}, \nabla g_{\rm i})_{\rm i}
  - (\partial_{\nu_{\rm i}}f_{\rm i}|_\Sigma, g_{\rm i}|_\Sigma)_\Sigma
  &= (-\Delta f_{\rm i},g_{\rm i})_{\rm i}, \\[0.5ex]
  (\nabla f_{\rm e}, \nabla g_{\rm e})_{\rm e}
  - (\partial_{\nu_{\rm e}}f_{\rm e}|_\Sigma, g_{\rm e}|_\Sigma)_\Sigma
  &= (-\Delta f_{\rm e},g_{\rm e})_{\rm e}.
\end{align*}
Using this and the relation $\alpha f|_\Sigma = \partial_{\nu_{\rm e}}f_{\rm e}|_\Sigma
  + \partial_{\nu_{\rm i}}f_{\rm i}|_\Sigma$
we obtain
\begin{align*}
  & \fra_{\delta,\alpha}[f,g]
  = (\nabla f,\nabla g) + (Vf,g) - \bigl(\alpha f|_\Sigma,g|_\Sigma\bigr)_\Sigma \\[0.5ex]
  &= (\nabla f_{\rm i},\nabla g_{\rm i})_{\rm i} + (\nabla f_{\rm e},\nabla g_{\rm e})_{\rm e}
  + (Vf,g) - \bigl(\partial_{\nu_{\rm i}}f_{\rm i}|_\Sigma,g_{\rm i}|_\Sigma\bigr)_\Sigma
  - \bigl(\partial_{\nu_{\rm e}}f_{\rm e}|_\Sigma,g_{\rm e}|_\Sigma\bigr)_\Sigma \\[0.5ex]
  &= (-\Delta f_{\rm i},g_{\rm i})_{\rm i} + (-\Delta f_{\rm e},g_{\rm e})_{\rm e} + (Vf,g)
  = \bigl((-\Delta+V)f,g\bigr).
\end{align*}
Now the first representation theorem (see \cite[Theorem~VI.2.1]{K95}) implies that
$f\in \dom \cA_{\delta,\alpha}$ and $\cA_{\delta,\alpha}f = -\Delta f + Vf$;
thus $A_{\delta,\alpha} \subset \cA_{\delta,\alpha}$. Since both operators $A_{\delta,\alpha}$
and $\cA_{\delta,\alpha}$ are self-adjoint, we conclude that $A_{\delta,\alpha} = \cA_{\delta,\alpha}$.
\end{proof}

% *******************************************************************
\subsection{Schr\"odinger operators with $\delta'$-interactions on
hypersurfaces: self-adjoint\-ness,  Krein's formula and $H^2$-regularity}
\label{sec:delta'}
% *******************************************************************

In this section we make use of quasi boundary triples to define and study the Schr\"odinger operator
$A_{\delta^\prime\!,\beta}$ associated with the formal differential expression
$\cL_{\delta,\alpha}=-\Delta+V-\beta\langle\delta^\prime_\Sigma,\cdot\,\rangle\,\delta^\prime_\Sigma$
in \eqref{expressions1}.
The methodology and presentation is very much the same as in the previous section.
We mention that to the best of our knowledge a systematic treatment of $\delta'$-potentials
on hypersurfaces is not contained elsewhere; see the list of open problems in \cite{E08}.

In analogy to \eqref{wtA} and \eqref{wtT} we define the symmetric extension
\begin{equation}
\label{whA}
  \wh A \defeq A_{\rm free} \cap  A_{\rm N, i, e}
  = \cL\upharpoonright\big\{ f\in H^2(\dR^n)\colon
   \partial_{\nu_{\rm i}}f_{\rm i}|_\Sigma  = \partial_{\nu_{\rm e}}f_{\rm e}|_\Sigma = 0\big\}
\end{equation}
of the orthogonal sum $A_{\rm i,e}$, defined in \eqref{aie}, which will serve as the underlying
symmetric operator for the quasi boundary triple in the next proposition, and the operator
\begin{equation}
\label{whT}
  \wh T \defeq
  T_{\rm i,e}\upharpoonright
  \bigl\{f_{\rm i} \oplus f_{\rm e} \in H^{3/2}_\Delta(\dR^n\backslash\Sigma) \colon
  \partial_{\nu_{\rm e}}f_{\rm e}|_\Sigma + \partial_{\nu_{\rm i}}f _{\rm i}|_\Sigma = 0 \bigr\}.
\end{equation}
We remark that the quasi boundary triple and Weyl function
below appear also implicitly in \cite[Section~4]{R09} in a different context.

\begin{prop}
\label{prop:qbt2}
Let the operators $\wh A$, $\wh T$, $A_{\rm D, i,e}$, $A_{\rm N,i,e}$ and $A_{\rm free}$
be as in~\eqref{whA}, \eqref{whT}, \eqref{ADANoplus1}, \eqref{ADANoplus2} and \eqref{Afree}, respectively, and
let $M_{\rm i}$ and $M_{\rm e}$ be the Weyl functions from Proposition~\ref{prop:qbt0}.
Then the following statements hold.
\begin{itemize}\setlength{\itemsep}{1.2ex}
%------------------------------------------------------------------------------
\item[\rm (i)]
The triple $\wh\Pi = \{L^2(\Sigma), \wh\Gamma_0,\wh\Gamma_1\}$, where
\begin{equation*}
  \wh\Gamma_0 f =
  \partial_{\rm \nu_{\rm e}} f_{\rm e}|_\Sigma,
  \quad
  \wh\Gamma_1f = f_{\rm e}|_\Sigma - f_{\rm i}|_\Sigma,
  \qquad
  f = f_{\rm i} \oplus f_{\rm e} \in \dom  \wh T,
\end{equation*}
is a quasi boundary triple for $\wh A^*$.
The restrictions of $\wh T$ to the kernels of the boundary
mappings are
\[
  \wh T \upharpoonright \ker\wh \Gamma_0 = A_{\rm N,i,e}
  \quad\text{and}\quad
  \wh T \upharpoonright \ker\wh \Gamma_1 = A_{\rm free},
 \]
and the ranges of the boundary mappings are
\[
  \ran \wh \Gamma_0 = L^2(\Sigma)
  \quad\text{and}\quad
  \ran \wh \Gamma_1 = H^1(\Sigma).
\]
%------------------------------------------------------------------------------
\item[\rm (ii)]
For $\lambda\in\rho(A_{\rm N,i,e})$ and $\varphi\in L^2(\Sigma)$ the problem
\begin{equation*}
(\cL-\lambda) f = 0,
    \qquad  \partial_{\nu_{\rm e}} f_{\rm e}|_\Sigma
    = -\partial_{\nu_{\rm i}} f_{\rm i}|_\Sigma
    = \varphi,
 \end{equation*}
has the unique solution $\wh\gamma(\lambda)\varphi\in H^{3/2}_\Delta(\dR^n\backslash\Sigma)$,
where $\wh\gamma$ is the $\gamma$-field associated with $\wh\Pi$.  Moreover, $\wh\gamma(\lambda)$ is bounded
from $L^2(\Sigma)$ to $L^2(\dR^n)$.
%------------------------------------------------------------------------------
\item[\rm (iii)]
For $\lambda\in\rho(A_{\rm N,i,e})$ the values $\wh M(\lambda)$ of the Weyl function
associated with $\wh \Pi$ are bounded operators
from $L^2(\Sigma)$ to $H^{1}(\Sigma)$ and compact operators in $L^2(\Sigma)$.
If, in addition, $\lambda\in\rho(A_{\rm free})$, then $\wh M(\lambda)$ is a bijective map
from $L^2(\Sigma)$ onto $H^1(\Sigma)$.
Moreover, the identity
\begin{equation}
\label{weyl2}
  \wh M(\lambda) = M_{\rm i}(\lambda)+M_{\rm e}(\lambda)
\end{equation}
holds for all $\lambda\in\rho(A_{\rm N,i,e})$.
%------------------------------------------------------------------------------
\end{itemize}
\end{prop}

%------------------------------------------------------------------------------
\begin{proof}
(i)
One can see that $\wh \Pi$ is a quasi boundary triple for $\wh A^*$ in a similar way as in the
proof of Proposition~\ref{prop:qbt1}\,(i). Basically, the same argumentation as before yields
that $\wh T\upharpoonright\ker\wh\Gamma_0 = A_{\rm N,i,e}$,
$\wh T\upharpoonright\ker\wh\Gamma_1 = A_{\rm free}$ and also that $\ran\wh\Gamma_0 = L^2(\Sigma)$,
$\ran\wh\Gamma_1\subset H^1(\Sigma)$. Further we show surjectivity of $\wh\Gamma_1$ onto $H^1(\Sigma)$.
Fix a function $\chi\in C_0^\infty(\RR^n)$ as in the proof of Proposition~\ref{prop:qbt1},
i.e.\ such that $\chi\equiv1$ on an open neighbourhood of $\overline{\Omega_{\rm i}}$.
Let ${\rm DL}$ be the double-layer potential associated with the hypersurface $\Sigma$ and
the differential expression $-\Delta+1$; see, e.g.\ \cite[Section 6]{McL00}
for the discussion of double-layer potentials. By \cite[Theorem~6.11, Theorem~6.12\,(ii)]{McL00}
for an arbitrary $\varphi\in H^1(\Sigma)$ the function $f \defeq \chi {\rm DL}\varphi$
belongs to $\dom \wh T$ and satisfies the condition
\[
  f_{\rm e}|_\Sigma - f_{\rm i}|_\Sigma = \varphi,
\]
hence $\wh \Gamma_1f = \varphi$, and thus $\ran\wh \Gamma_1 = H^1(\Sigma)$.

(ii)--(iii)
The properties of the $\gamma$-field $\wh\gamma$ and the Weyl function $\wh M$ follow
from Proposition~\ref{gammaprop} in the same way as in the proof of Proposition~\ref{prop:qbt1}\,(ii)--(iii).
We only verify the identity \eqref{weyl2}. For this let
$\lambda\in\rho(A_{\rm N,i,e})$, so that the operators
$M_{\rm i}(\lambda)$,  $M_{\rm e}(\lambda)$ and $\wh M(\lambda)$ all exist; cf.\ Proposition~\ref{prop:qbt0}.
If $\wh M(\lambda) \varphi = \psi$ for some $\varphi\in L^2(\Sigma)$ and $\psi \in H^1(\Sigma)$, then
there exists $f = f_{\rm i}\oplus f_{\rm e}\in \ker (\wh T-\lambda)$ such that
\[
 \wh \Gamma_0 f = \varphi\quad\text{and}\quad\wh\Gamma_1f = \psi.
\]
As $f_{\rm i}\in \ker(T_{\rm i} -\lambda)$ and $f_{\rm e} \in \ker(T_{\rm e} - \lambda)$, we have
\begin{align*}
  \Gamma_{1,\rm i}f_{\rm i} &= M_{\rm i}(\lambda)\Gamma_{0,\rm i}f_{\rm i} = -M_{\rm i}(\lambda)\varphi,\\
  \Gamma_{1,\rm e}f_{\rm e} &= M_{\rm e}(\lambda)\Gamma_{0,\rm e}f_{\rm e} = M_{\rm e}(\lambda)\varphi,
\end{align*}
and hence
\[
  \wh M(\lambda)\varphi = f_{\rm e}|_\Sigma - f_{\rm i}|_\Sigma
  = M_{\rm e}(\lambda)\varphi + M_{\rm i}(\lambda)\varphi.
\]
Since this is true for arbitrary $\varphi\in L^2(\Sigma)$, relation \eqref{weyl2} follows.
\end{proof}

%-------------------------------------------------------------------------------------
\begin{rem}
Assume for simplicity that the potential $V$
in the differential expression $\cL$ in \eqref{Schrodinger} is identically equal to zero.
Note that the problem in (ii) is decoupled into an interior and an exterior problem.
Let, as in Remark~\ref{acoustic},  $G_\lambda$ be the integral kernel of the
resolvent of $A_{\rm free}$. Then for all $\psi\in C^\infty(\Sigma)$
\[
\big(\wh\gamma(\lambda)\wh M(\lambda)^{-1}\psi\big)(x) = \int_\Sigma \big[\partial_{\nu_{\rm i}(y)}G_\lambda(x,y)\big] \psi(y)d\sigma_y, \quad x\in\dR^n\backslash\Sigma,
\]
and
\[
\big(\wh M(\lambda)^{-1}\psi\big)(x) = -\partial_{\nu_{\rm i}(x)}\int_\Sigma \big[\partial_{\nu_{\rm i}(y)}G_\lambda(x,y)\big] \psi(y)d\sigma_y,\quad x\in\Sigma,
\]
where $\partial_{\nu_{\rm i}(x)}$ and $\partial_{\nu_{\rm i}(y)}$ are the normal derivatives with respect to the first and second arguments with normals pointing outwards of $\Omega_{\rm i}$, and $\sigma_y$ is the natural Lebesgue measure on $\Sigma$.  
Note that the operator $\wh\gamma(\lambda)\wh M(\lambda)^{-1}$ is, roughly speaking, an extension of the acoustic
double-layer potential for the Helmholtz equation, see, e.g.\ \cite[Chapter 6]{McL00}. 
The representation of $\wh M(\lambda)^{-1}$, given above, appears also in \cite{R09} in a slightly different context.
\end{rem}

%-------------------------------------------------------------------------------------

We repeat the definition of the Schr\"odinger operator with $\delta^\prime$-potential
from the introduction and relate it to the quasi boundary triple $\wh\Pi$.

\begin{dfn}\label{def:deltapr}
For a real-valued function $\beta$ such that $1/\beta\in L^\infty(\Sigma)$
the Schr\"odinger operator with $\delta'$-potential on the hypersurface $\Sigma$ and strength $\beta$
is defined as follows:
\begin{equation*}
  A_{\delta'\!,\beta} = \wh T \upharpoonright \ker(\wh \Gamma_1-\beta\wh\Gamma_0),
\end{equation*}
which is equivalent to
\begin{equation}
\label{eq:delta'}
\begin{split}
  A_{\delta^\prime\!,\beta}f \defequ -\Delta f+ Vf, \\
  \dom A_{\delta^\prime\!,\beta} \defequ \Biggl\{f\in H^{3/2}_\Delta(\dR^n\backslash\Sigma):\begin{matrix}
  \partial_{\nu_e}f_{\rm e}\vert_\Sigma=-\partial_{\nu_{\rm i}}f_{\rm i}\vert_\Sigma\\
  \beta \partial_{\nu_{\rm e}}f_{\rm e}|_\Sigma=f_{\rm e}|_\Sigma-f_{\rm i}|_\Sigma\end{matrix}  \Biggr\}.
\end{split}
\end{equation}
\end{dfn}

The definition of $A_{\delta'\!,\beta}$ is compatible with the definition of a
point $\delta'$-interaction in the one-dimensional case~\cite[Section~I.4]{AGHH05}, \cite{AK99}
and the definition of the operator with $\delta'$-potentials on spheres given in~\cite{AGS87, S88-2}.
Note that, in contrast to the domain of $A_{\delta,\alpha}$, the domain of $A_{\delta'\!,\beta}$ is
not contained in $H^1(\dR^n)$.

%---Fig. 2-----------------------------------------------------------------
\begin{figure}[H]
\begin{displaymath}
   \xymatrix{
& & \text{$A_{\rm N,i,e}$}\ar@{}[dr] | {\begin{turn}{-45}$\subset$\end{turn}}\\
%----------------------------------------------------
\text{$A_{\rm i,e}$} \ar@{}[r] |{\begin{turn}{0}$\subset$\end{turn}} &
\text{$\wh A$}
\ar@{}[r] |{\begin{turn}{0}$\subset$\end{turn}}
\ar@{}[ur] | {\begin{turn}{45}$\subset$\end{turn}}
\ar@{}[dr] | {\begin{turn}{-45}$\subset$\end{turn}} &
\text{$A_{\delta'\!,\beta}$} \ar@{}[r] |{\begin{turn}{0}$\subset$\end{turn}} &
\text{$\wh T$} \ar@{}[r] |{\begin{turn}{0}$\subset$\end{turn}}&
\text{$T_{\rm i,e}$} \\
%-----------------------------------------------------------------
&& \text{\raisebox{4ex}{$A_{\rm free}$}}\ar@{}[ur] | {\begin{turn}{45}$\subset$\end{turn}}}
\qquad \begin{aligned} \ov{\wh T} &= \wh A^* \\[1ex] \ov {T}_{\rm i,e} &= A_{\rm i,e}^* \end{aligned}
\end{displaymath}
\caption{This figure shows how the operator $A_{\delta'\!,\beta}$ is related to the other operators
studied in this section. The operators $A_{\rm N,i,e}$, $A_{\delta'\!,\beta}$ and $A_{\rm free}$
are self-adjoint.}
\label{fig2}
\end{figure}
%---Fig. 2-----------------------------------------------------------------

The next theorem is the counterpart of Theorem~\ref{thm:delta} and can be proved in the same way.
Theorem~\ref{thm:delta'} shows the self-adjointness of $A_{\delta'\!,\beta}$ and
provides a factorization for the resolvent difference of $A_{\delta'\!,\beta}$
and $A_{\rm N,i,e}$ via Krein's formula and a variant of the Birman--Schwinger principle.
The first item of the next theorem is part of Theorem A in the introduction.

%----------------------------------------
\begin{thm}
\label{thm:delta'}
Let $A_{\delta'\!,\beta}$ be as above and let $A_{\rm N,i,e}$ be the self-adjoint operator
defined in~\eqref{ADANoplus2}. Let $\wh \gamma$ and $\wh M$ be the $\gamma$-field
and the Weyl function associated with the quasi boundary triple $\wh \Pi$ from
Proposition~\ref{prop:qbt2}.  Then the following statements hold.
\begin{itemize}\setlength{\itemsep}{1.2ex}
%--Self-adjointness-------------------------
\item[(i)]
The operator $A_{\delta'\!,\beta}$ is self-adjoint in the Hilbert space $L^2(\dR^n)$.
%--Krein's formula--------------------------------
\item[(ii)]
For all $\lambda\in\rho(A_{\delta'\!,\beta})\cap\rho(A_{\rm N,i,e})$ the following
Krein formula holds:
\[
(A_{\delta'\!,\beta} - \lambda)^{-1} -(A_{\rm N,i,e} - \lambda)^{-1} =
\wh\gamma(\lambda)\bigl(I - \beta^{-1}\wh M(\lambda)\bigr)^{-1}\beta^{-1}\,\wh\gamma(\ov\lambda)^*,
\]
where $(I- \beta^{-1}\wh M(\lambda))^{-1}\in\cB(L^2(\Sigma))$.
\item[\rm (iii)] For all $\lambda\in\dR\setminus\sigma(A_{\rm N,i,e})$ we have
\[\lambda\in\sigma_p(A_{\delta^\prime\!,\beta})\quad\Longleftrightarrow \quad
0\in\sigma_p\bigl(I - \beta^{-1}\wh M(\lambda)\bigr)\]
and $\dim\ker(A_{\delta^\prime\!,\beta}-\lambda)=\dim\ker (I - \beta^{-1}\wh M(\lambda))$.
\end{itemize}
\end{thm}

The next theorem gives assumptions on $\beta$ which ensure that the domain of
the self-adjoint operator $A_{\delta'\!,\beta}$ has $H^2$-regularity.
This theorem is the second part of Theorem~B in the introduction.

%---Theorem: H^2-regularity----------------------------
\begin{thm}
\label{thm:H2delta'}
Let $A_{\delta'\!,\beta}$ be the self-adjoint Schr\"odinger operator in
Definition~\ref{def:deltapr} and assume, in addition, that
the function $\beta\colon\Sigma\rightarrow\dR$ is such
that $1/\beta\in W^{1,\infty}(\Sigma)$. Then $\dom A_{\delta'\!,\beta}$ is
contained in $H^2(\dR^n\backslash\Sigma)$.
\end{thm}

\begin{proof}
The proof proceeds as the proof of Theorem~\ref{thm:H2delta}
with $A_{\delta,\alpha}$, $A_{\rm free}$, $\wt T$, $\wt \Gamma_0$, $\wt \Gamma_1$ and $\alpha$
replaced by
$A_{\delta'\!,\beta}$, $A_{\rm N,i,e}$, $\wh T$, $\wh \Gamma_0$, $\wh \Gamma_1$ and $\beta^{-1}$, respectively.
Instead of the decomposition~\eqref{decomposition2} one has to use the decomposition
\[
  \dom \wh T = \dom A_{\rm N,i,e} \dotplus \ker(\wh T-\lambda),\qquad\lambda\in\dC\setminus\dR.
\]
\end{proof}

% *******************************************************************
\subsection{Semi-boundedness and point spectra}
\label{sec:finiteness}
% *******************************************************************

In this section we show that the self-adjoint operators $A_{\delta,\alpha}$
and $A_{\delta^\prime\!,\beta}$ are lower semi-bounded, and that in the
case $V\equiv 0$ their negative spectra are finite.
We recall some preparatory facts on semi-bounded quadratic forms first.

% -------------------------------------------------------------------
\begin{dfn}
For a (not necessarily closed or semi-bounded) quadratic form $\frq$ in a
Hilbert space $\cH$  we define \emph{the number of negative squares}
$\kappa_-(\frq)$ by
\[
\begin{split}
& \kappa_-(\frq) \defeq \sup\big\{\dim F\colon\! F \text{ linear subspace
  of $\dom \frq$}\\
&\qquad\qquad\qquad\qquad\qquad\qquad\qquad\qquad\text{such that}\,\forall\, f\in F\setminus\{0\}\colon\! \frq[f] < 0\big\}.
\end{split}
\]
\end{dfn}
Assume that $A$ is a (not necessarily semi-bounded) self-adjoint operator
in a Hilbert space $\cH$ with the corresponding spectral measure $E_A(\cdot)$.
Define the possibly non-closed quadratic form $\frs_A$ by
\[
 \frs_A[f] \defeq (Af,f)_{\cH}, \quad \dom \frs_A \defeq \dom A.
\]
If, in addition, $A$ is semi-bounded, then by \cite[Theorem~VI.1.27]{K95}
the form $\frs_A$ is closable, and we denote its closure by $\ov{\frs_A}$.
According to the spectral theorem for self-adjoint operators
and~\cite[10.2 Theorem 3]{BS87}
\begin{equation}
\label{BS}
  \dim \ran E_A(-\infty,0) = \kappa_-(\frs_A) = \kappa_-(\ov{\frs_A}).
\end{equation}
In particular, if $\kappa_-(\frs_A)$ is finite, then the self-adjoint
operator $A$ has finitely many negative eigenvalues with finite multiplicities.

In the case $V\equiv 0$ we write $-\Delta_{\delta,\alpha}$,
$-\Delta_{\delta'\!,\beta}$ and $-\Delta_{\rm free}$
instead of $A_{\delta,\alpha}$, $A_{\delta'\!,\beta}$ and $A_{\rm free}$.
Now we are ready to formulate and prove the main results of this section.
The next theorem is part of Theorem~A in the introduction.
We mention that finiteness of the negative spectrum in the case of
$\delta$-potentials on hypersurfaces was also shown in~\cite{BEKS94}
by other methods.

%-------------------------------------------------------------------
\begin{thm}
\label{thm:smb}
Let $\alpha, \beta \colon \Sigma\rightarrow \dR$ be such
that $\alpha, 1/\beta\in L^\infty(\Sigma)$ and let the self-adjoint
operators $-\Delta_{\delta,\alpha}$ and $-\Delta_{\delta'\!,\beta}$ be as
above.
Then the following statements hold.
\begin{itemize}\setlength{\itemsep}{1.2ex}
%-------------------------------------------------------------------
\item[(i)]
$\sess(-\Delta_{\delta,\alpha}) = \sess(-\Delta_{\delta'\!,\beta}) =
[0,\infty)$.
%-------------------------------------------------------------------
\item[(ii)]
The self-adjoint operators $-\Delta_{\delta,\alpha}$ and
$-\Delta_{\delta'\!,\beta}$
have finitely many negative eigenvalues with finite multiplicities.
%-------------------------------------------------------------------
\end{itemize}
\end{thm}
%-------------------------------------------------------------------

\begin{proof}
(i)
According to Theorem~\ref{thm:S_infty2} in Section~\ref{sec:pwrdiff} below
the resolvent difference of the self-adjoint operators
$-\Delta_{\delta,\alpha}$
and $-\Delta_{\rm free}$ is compact; thus
\[
 \sess(-\Delta_{\delta,\alpha}) = \sess(-\Delta_{\rm free}) = [ 0,
\infty).
\]
Analogously, according to Theorem~\ref{thm:S_infty4} below the resolvent
difference
of the self-adjoint operators $-\Delta_{\delta'\!,\beta}$ and
$-\Delta_{\rm free}$ is also compact.
Hence
\[
 \sess(-\Delta_{\delta'\!,\beta}) = \sess(-\Delta_{\rm free}) = [0,\infty).
\]

(ii)
Let us introduce the (in general non-closed) quadratic forms
\begin{alignat*}{2}
 \frs_{-\Delta_{\delta,\alpha}}[f] \defequ
\bigl(-\Delta_{\delta,\alpha}f,f\bigr),\quad &
 \dom (\frs_{-\Delta_{\delta,\alpha}}) \defequ \dom
(-\Delta_{\delta,\alpha}), \\[0.5ex]
 \frs_{-\Delta_{\delta'\!,\beta}}[f] \defequ
\bigl(-\Delta_{\delta^\prime\!,\beta}f,f\bigr), \quad &
 \dom (\frs_{-\Delta_{\delta'\!,\beta}}) \defequ \dom
(-\Delta_{\delta^\prime\!,\beta}).
\end{alignat*}
Applying the first Green's identity~\eqref{half_green} to these
expressions and taking
the definitions \eqref{eq:delta}, \eqref{eq:delta'} of the domains of the
operators $-\Delta_{\delta,\alpha}$, $-\Delta_{\delta^\prime\!,\beta}$
into account we obtain
\begin{align*}
  \frs_{-\Delta_{\delta,\alpha}}[f]
  &= \bigl(-\Delta f_{\rm i},f_{\rm i}\bigr)_{\rm i} + \bigl(-\Delta f_{\rm e},
  f_{\rm e}\bigr)_{\rm e} \\[0.5ex]
  &= \bigl(\nabla f_{\rm i},\nabla f_{\rm i}\bigr)_{\rm i}
  - \bigl(\partial_{\nu_{\rm i}}f_{\rm i}|_\Sigma,f_{\rm i}|_\Sigma\bigr)_\Sigma
  + \bigl(\nabla f_{\rm e},\nabla f_{\rm e}\bigr)_{\rm e}
  - \bigl(\partial_{\nu_{\rm e}}f_{\rm e}|_\Sigma,f_{\rm e}|_\Sigma\bigr)_\Sigma \\[0.5ex]
  &= \bigl(\nabla f,\nabla f\bigr) - \bigl(\alpha f|_\Sigma,f|_\Sigma\bigr)_\Sigma
\end{align*}
and
\begin{align*}
  \frs_{-\Delta_{\delta'\!,\beta}}[f]
  &= \bigl(-\Delta f_{\rm i},f_{\rm i}\bigr)_{\rm i} + \bigl(-\Delta f_{\rm e},
  f_{\rm e}\bigr)_{\rm e} \\[0.5ex]
  &= \bigl(\nabla f_{\rm i},\nabla f_{\rm i}\bigr)_{\rm i}
  \!-\! \bigl(\partial_{\nu_{\rm i}}f_{\rm i}|_\Sigma,f_{\rm i}|_\Sigma\bigr)_\Sigma
  + \bigl(\nabla f_{\rm e},\nabla f_{\rm e}\bigr)_{\rm e}
  - \bigl(\partial_{\nu_{\rm e}}f_{\rm e}|_\Sigma,f_{\rm e}|_\Sigma\bigr)_\Sigma \\[0.5ex]
  &= (\nabla f,\nabla f)\! +\! \bigl(\beta^{-1}(f_{\rm e}|_\Sigma\!-\!f_{\rm i}|_\Sigma),
  f_{\rm i}|_\Sigma\bigr)_\Sigma\!
  -\!\bigl(\beta^{-1}(f_{\rm e}|_\Sigma\!-\!f_{\rm i}|_\Sigma),f_{\rm e}|_\Sigma\bigr)_\Sigma \\[0.5ex]
  &= \bigl(\nabla f, \nabla f \bigr) -\bigl(\beta^{-1}
  (f_{\rm e}|_\Sigma-f_{\rm i}|_\Sigma),f_{\rm e}|_\Sigma-f_{\rm i}|_\Sigma\bigr)_\Sigma.
\end{align*}
For a bounded function $\sigma\colon\Sigma\rightarrow\dR$ define the
quadratic form $\frq_\sigma$
\[
\begin{split}
 \frq_\sigma[f] \defequ \big(\nabla f, \nabla f\big) -
 \big(\sigma f_{\rm i}|_\Sigma,f_{\rm i}|_\Sigma\big)_\Sigma
 - \big(\sigma f_{\rm e}|_\Sigma,f_{\rm e}|_\Sigma\big)_\Sigma,\\
 \dom \frq_\sigma \defequ H^1(\dR^n\backslash\Sigma).
\end{split}
\]
It follows from \cite[Theorem~6.9]{B62} (cf.\ the proof of
Proposition~\ref{prop:deltapform} below) that the form $\frq_\sigma$ is
closed and semi-bounded, and the self-adjoint operator
corresponding to $\frq_\sigma$ has finitely many negative eigenvalues
with finite multiplicities. Thus, by~\eqref{BS}, we have
$\kappa_-(\frq_\sigma)< \infty$.
It can easily be checked that
\[
 \dom( \frs_{-\Delta_{\delta,\alpha}}) \subset \dom (\frq_{|\alpha|/2})
\! \quad\text{and}\quad\!
 \forall\, f\in \dom(\frs_{-\Delta_{\delta,\alpha}})\colon
 \frs_{-\Delta_{\delta,\alpha}}[f] \ge  \frq_{|\alpha|/2}[f].
\]
Using the inequality  $|a-b|^2 \le 2(|a|^2+|b|^2)$ for complex numbers
$a,b$ we obtain
\[
 \dom (\frs_{-\Delta_{\delta'\!,\beta}}) \subset \dom (\frq_{2/|\beta|})
\!\! \quad\text{and}\quad\!\!
 \forall\, f\in \dom (\frs_{-\Delta_{\delta'\!,\beta}})\colon
 \frs_{-\Delta_{\delta'\!,\beta}}[f]\ge  \frq_{2/|\beta|}[f].
\]
These observations yield that
\[
 \kappa_-(\frs_{-\Delta_{\delta,\alpha}})\le
\kappa_-(\frq_{|\alpha|/2}) < \infty
 \quad\text{and}\quad
 \kappa_-(\frs_{-\Delta_{\delta'\!,\beta}}) \le
\kappa_-(\frq_{2/|\beta|}) < \infty.
\]
From this and \eqref{BS} it follows that the negative spectra of
$-\Delta_{\delta,\alpha}$
and $-\Delta_{\delta'\!,\beta}$ are finite.
\end{proof}

In the following proposition the (closed) sesquilinear form $\mathfrak
t_{-\Delta_{\delta'\!,\beta}}$ which induces
the self-adjoint operator $-\Delta_{\delta'\!,\beta}$ is determined. This
was posed as an open problem in \cite[7.2]{E08}.
Note that, by the first representation theorem,
$\mathfrak t_{-\Delta_{\delta'\!,\beta}}$ is the closure of the form
$$\frs_{-\Delta_{\delta'\!,\beta}}[f,g]=(-\Delta_{\delta'\!,\beta}f,g),
\qquad f,g\in\dom(-\Delta_{\delta'\!,\beta}),$$
defined in the proof of Theorem~\ref{thm:smb}.
For completeness we mention that Proposition~\ref{prop:deltapform} extends naturally
to the Schr\"odinger operator $A_{\delta^\prime\!,\beta}$ with non-trivial $V\in L^\infty(\dR^n)$
and the corresponding quadratic form.

\begin{prop}\label{prop:deltapform}
The sesquilinear form
\begin{equation*}
  \fra_{\delta'\!,\beta}[f,g]
  \defeq \bigl(\nabla f, \nabla g \bigr) - \bigl(\beta^{-1}
  (f_{\rm e}|_\Sigma-f_{\rm i}|_\Sigma),g_{\rm e}|_\Sigma-g_{\rm i}|_\Sigma\bigr)_\Sigma
\end{equation*}
defined for $f,g\in H^1(\dR^n\backslash\Sigma)$
is symmetric, closed and semi-bounded from below. The self-adjoint
operator corresponding to
$\fra_{\delta'\!,\beta}$ is $-\Delta_{\delta'\!,\beta}$, i.e.\
\begin{equation*}
  (-\Delta_{\delta'\!,\beta}f,g)= \fra_{\delta'\!,\beta}[f,g]
\end{equation*}
holds for all $f\in\dom(-\Delta_{\delta'\!,\beta})$ and $g\in
H^1(\dR^n\backslash\Sigma)$.
\end{prop}

\begin{proof}
Since $\beta$ is a real-valued function, it follows that the form
$\fra_{\delta'\!,\beta}$ is symmetric.
In order to show that it is closed and semi-bounded, we consider the forms
\begin{equation*}
  \fra[f,g] \defeq (\nabla f,\nabla g)\quad\text{and}\quad\mathfrak
  \fra^\prime[f,g]\defeq -\bigl(\beta^{-1}
  (f_{\rm e}|_\Sigma-f_{\rm i}|_\Sigma),g_{\rm e}|_\Sigma
  -g_{\rm i}|_\Sigma\bigr)_\Sigma
\end{equation*}
on $H^1(\dR^n\backslash\Sigma)$, so that
$\fra_{\delta'\!,\beta}=\fra+\fra^\prime$
holds. Note that $\mathfrak t$ is closed and non-negative.
Let $t\in(\tfrac{1}{2},1)$ be fixed. Since the trace map is continuous,
there exists $c_t>0$ such that
$\Vert f_{\rm i}\vert_\Sigma\Vert_{H^{t-1/2}(\Sigma)}\leq c_t \Vert f_{\rm i}\Vert_{H^t(\Omega_{\rm i})}$
is valid for all $f_{\rm i}\in H^t(\Omega_{\rm i})$.
Hence it follows from Ehrling's lemma that for every $\eps>0$ there
exists a constant $C_{\rm i}(\eps)$ such that
\begin{equation}\label{esti}
  \Vert f_{\rm i}\vert_\Sigma \Vert_\Sigma\leq c_t\Vert f_{\rm i}
  \Vert_{H^t(\Omega_{\rm i})}
  \leq
  \eps\Vert f_{\rm i}\Vert_{H^1(\Omega_{\rm i})}  + C_{\rm i}(\eps)\Vert
  f_{\rm i}\Vert_{L^2(\Omega_{\rm i})}
\end{equation}
holds for all $f_{\rm i}\in H^1(\Omega_{\rm i})$. We decompose the
exterior domain in the form
$\Omega_{\rm e}= \Omega_{\rm e,1}\cup\overline\Omega_{\rm e,2}$, where
$\Omega_{\rm e,1}$ is bounded, $\Omega_{\rm e,2}$ is unbounded, and
the $C^\infty$-boundary of $\Omega_{\rm e,1}$ is the disjoint union of $\Sigma$ and
$\partial\Omega_{\rm e,2}$. The restriction of a function $f_{\rm e}$
to $\Omega_{\rm e,1}$ is denoted by $f_{\rm e,1}$. Then
again the continuity of the trace map and Ehrling's lemma show that for
every $\eps>0$ there exists a constant $C_{\rm e}(\eps)$ such that
\begin{equation}\label{este}
\begin{split}
  \Vert f_{\rm e}\vert_\Sigma \Vert_\Sigma  = \Vert f_{\rm e,1}\vert_\Sigma \Vert_\Sigma
  & \leq  \Vert f_{\rm e,1}
  \vert_{\partial\Omega_{\rm e,1}} \Vert_{L^2(\partial\Omega_{\rm e,1})}\\
  &\leq
  \eps\Vert f_{\rm e,1}\Vert_{H^1(\Omega_{\rm e,1})} + C_{\rm e}(\eps)
  \Vert f_{\rm e,1}\Vert_{L^2(\Omega_{\rm e,1})}\\
  &\leq \eps\Vert f_{\rm e}\Vert_{H^1(\Omega_{\rm e})} + C_{\rm e}(\eps)
  \Vert f_{\rm e}\Vert_{L^2(\Omega_{\rm e})}
\end{split}
\end{equation}
holds for all $f_{\rm e}\in H^1(\Omega_{\rm e})$.
The estimates \eqref{esti} and \eqref{este} yield that the form
$\fra^\prime$
is bounded with respect to $\fra$ with form
bound $<1$, and hence $\fra_{\delta'\!,\beta}=\fra+\fra^\prime$ 
is closed and semi-bounded by \cite[Theorem~VI.1.33]{K95}.
The remaining statement follows from \cite[Theorem~VI.2.1]{K95} and similar arguments as in the proof of Proposition~\ref{prop:exner}.
\end{proof}

Items (i) and (ii) in the next theorem are part of Theorem A in the
introduction.

\begin{thm}\label{thm:smb2}
Let $\alpha,\beta\colon\Sigma\rightarrow\dR$ be such that
$\alpha,1/\beta\in L^\infty(\Sigma)$
and let $V\in L^\infty(\dR^n)$ be a real-valued potential. Moreover, let the
self-adjoint operators $A_{\delta,\alpha}$, $A_{\delta'\!,\beta}$, and
$A_{\rm free}$
be as in \eqref{eq:delta}, \eqref{eq:delta'} and \eqref{Afree},
respectively.
Then the following statements hold.
\begin{itemize}\setlength{\itemsep}{1.2ex}
%-------------------------------------------------------------------
\item[(i)]
$\sess(A_{\delta,\alpha}) = \sess(A_{\delta'\!,\beta}) = \sess(A_{\rm free})$.
%-------------------------------------------------------------------
\item[(ii)]
Both self-adjoint operators $A_{\delta,\alpha}$ and $A_{\delta'\!,\beta}$
are lower semi-bounded.

\end{itemize}
\end{thm}
%-------------------------------------------------------------------

\begin{proof}
(i)
The equality of the essential spectra follows from the stability of the
essential spectrum under compact perturbations and
Theorems~\ref{thm:S_infty2}
and \ref{thm:S_infty4} below.

(ii)
By Theorem~\ref{thm:smb}\,(ii) the operators $-\Delta_{\delta,\alpha}$
and $-\Delta_{\delta'\!,\beta}$ are bounded from below. The operator of
multiplication with the function $V$ is bounded and self-adjoint.
Thus the operators $A_{\delta,\alpha}$ and $A_{\delta'\!,\beta}$ are
bounded from below.
\end{proof}

% *******************************************************************
% *******************************************************************
\section{Resolvent power differences in $\sS_{p,\infty}$-classes,
existence and completeness of wave operators}
\label{sec:4}
% *******************************************************************
% *******************************************************************

In this section we compare the powers of the resolvents of the singularly perturbed
self-adjoint Schr\"odinger operators $A_{\delta,\alpha}$ and $A_{\delta^\prime\!,\beta}$
with the powers of the resolvents of the unperturbed Schr\"odinger operator $A_{\rm free}$.
This leads to singular value estimates, which have a long tradition
in the analysis of elliptic differential operators, cf.\ \cite{B62,BS79,BS80,G84,K67}
and the recent contributions \cite{BLLLP10,BLL10,G11,G11-2,M10} for more details.
In this section we prove Theorem C and Theorem D from the Introduction in a slightly
stronger form. 
%in terms of weak Schatten--von Neumann classes. These operator
%ideals are also used, e.g.\ on counting negative eigenvalues of Schr\"odinger
%operators~\cite{C77,S76}.

% *******************************************************************
\subsection{Elliptic regularity and some preliminary $\sS_{p,\infty}$-estimates}
\label{sec:pwrprelim}
% *******************************************************************

In this section we first provide a typical regularity result for the
functions $(A_{\rm free}-\lambda)^{-1}f$ and $(A_{\rm N,i,e}-\lambda)^{-1}f$
if $f$ and $V$ satisfy some additional local smoothness assumptions. This fact is then used to obtain
estimates for the singular values of certain compact operators arising in
the representations of the resolvent power differences of the self-adjoint operators
$A_{\delta,\alpha}$, $A_{\delta'\!,\beta}$, $A_{\rm free}$ and $A_{\rm N,i,e}$.
In the next lemma we make use of the local Sobolev spaces $W^{k,\infty}_\Sigma(\dR^n)$,
$W^{k,\infty}_\Sigma(\dR^n\backslash\Sigma)$ and $H^k_\Sigma(\dR^n)$,
$H^k_\Sigma(\dR^n\backslash\Sigma)$ defined in Section~\ref{sec:locspaces}.

\begin{lem}\label{lem:smoothing}
Let $A_{\rm free}$ and $A_{\rm N,i,e}$ be the self-adjoint operators from \eqref{Afree}
and \eqref{ADANoplus2}, respectively, and let
$m\in\dN_0$. Then the following assertions hold.
\begin{itemize}\setlength{\itemsep}{1.2ex}
\item[(i)] If $V\in W^{m,\infty}_\Sigma(\dR^n)$, then, for
all $\lambda\in\rho(A_{\rm free})$ and $k = 0,1,\dots,m$,
\begin{equation*}
f\in H^k_\Sigma(\dR^n) \quad \Longrightarrow\quad (A_{\rm free}-\lambda)^{-1}f\in  H^{k+2}_\Sigma(\dR^n).
\end{equation*}
\item[(ii)] If $V\in W^{m,\infty}_\Sigma(\dR^n\backslash\Sigma)$, then, for
all $\lambda\in\rho(A_{\rm N,i,e})$ and $k = 0,1,\dots,m$,
\begin{equation*}
  f\in H^k_\Sigma(\dR^n\backslash\Sigma) \quad \Longrightarrow\quad
  (A_{\rm N,i,e}-\lambda)^{-1}f\in H^{k+2}_\Sigma(\dR^n\backslash\Sigma).
\end{equation*}
\end{itemize}
\end{lem}

\begin{proof}
We verify only assertion (i); the proof of (ii) is similar and left to the reader.
We proceed by induction with respect to $k$. For $k=0$ the statement is an immediate consequence of
$H^0_\Sigma(\dR^n)=L^2(\dR^n)$ and $\dom A_{\rm free}=H^2(\dR^n)$. Suppose now that the
implication in (i) is true for some
fixed $k< m$ and let $f\in H^{k+1}_\Sigma(\dR^n)$. Then, in particular, $f\in H^k_\Sigma(\dR^n)$ and hence
\[
  u \defeq (A_{\rm free}-\lambda)^{-1}f\in  H^{k+2}_\Sigma(\dR^n)\subset H^{k+1}_\Sigma(\dR^n)
\]
by assumption.
As $k+1\leq m$ and $V\in W^{m,\infty}_\Sigma(\dR^n)$, it follows from \eqref{W1Hk}
that $Vu\in H^{k+1}_\Sigma(\dR^n)$.
Therefore $f - Vu \in H^{k+1}_\Sigma(\dR^n)$, and since
the function $u$ satisfies the differential equation
\begin{equation*}
\label{pde1}
-\Delta u -\lambda u = f - Vu \quad\text{in}~\dR^n
\end{equation*}
standard results on elliptic regularity yield  $u\in H_\Sigma^{k+3}(\dR^n)$;
see, e.g. \cite[Theorem~4.18]{McL00}.
\end{proof}

An application of the previous lemma yields the following proposition, in which we provide
certain preliminary $\sS_{p,\infty}$-estimates
that are useful in the proofs of our main results in the next subsection.

\begin{prop}\label{prop:prelim}
Let $A_{\rm free}$ and $A_{\rm N,i,e}$ be the self-adjoint operators from \eqref{Afree}
and \eqref{ADANoplus2}, respectively, and let $\wt\gamma$ and $\wh\gamma$ be the $\gamma$-fields
from Propositions~\ref{prop:qbt1} and~\ref{prop:qbt2}, respectively.
Then for a fixed $m\in\dN_0$ the following statements hold.
\begin{itemize}\setlength{\itemsep}{1.2ex}
\item[(i)] If $V\in W^{2m,\infty}_{\Sigma}(\dR^n)$, then, for
all $\lambda,\mu\in\rho(A_{\rm free})$ and $k=0,1,\dots,m$,
\begin{equation*}
 \begin{split}
  &\text{\rm (a)}\,\,\,\,\,\wt\gamma(\mu)^*(A_{\rm free} -\lambda)^{-k}\in \sS_{\frac{n-1}{2k+3/2},\infty}\bigl(L^2(\dR^n), L^2(\Sigma)\bigr),\\
  &\text{\rm (b)}\,\,\,\,\,\wt\gamma(\mu)^*(A_{\rm free} -\lambda)^{-k}\in \sS_{\frac{n-1}{2k+1/2},\infty}\bigl(L^2(\dR^n), H^1(\Sigma)\bigr),\\
  &\text{\rm (c)}\,\,\,\,\,(A_{\rm free} -\lambda)^{-k}\wt\gamma(\mu)\in \sS_{\frac{n-1}{2k+3/2},\infty}\bigl( L^2(\Sigma), L^2(\dR^n)\bigr).
 \end{split}
\end{equation*}
\item[(ii)] If $V\in W^{2m,\infty}_{\Sigma}(\dR^n\backslash\Sigma)$, then, for
all $\lambda,\mu\in\rho(A_{\rm N,i,e})$ and $k=0,1,\dots,m$,
\begin{equation*}
 \begin{split}
  &\text{\rm (a)}\,\,\,\,\,\wh\gamma(\mu)^*(A_{\rm N,i,e} -\lambda)^{-k}\in \sS_{\frac{n-1}{2k+3/2},\infty}\bigl(L^2(\dR^n), L^2(\Sigma)\bigr),\\
  &\text{\rm (b)}\,\,\,\,\,\wh\gamma(\mu)^*(A_{\rm N,i,e} -\lambda)^{-k}\in \sS_{\frac{n-1}{2k+1/2},\infty}\bigl(L^2(\dR^n), H^1(\Sigma)\bigr),\\
  &\text{\rm (c)}\,\,\,\,\,(A_{\rm N,i,e} -\lambda)^{-k}\wh\gamma(\mu)\in \sS_{\frac{n-1}{2k+3/2},\infty}\bigl( L^2(\Sigma), L^2(\dR^n)\bigr).
 \end{split}
\end{equation*}
\end{itemize}
\end{prop}

\begin{proof}
We prove assertion (i); the proof of (ii) is analogous.
As $\ran (A_{\rm free} -\lambda)^{-1} =
\dom A_{\rm free} = H^2(\dR^n)\subset H^2_\Sigma(\dR^n)$ we conclude from Lemma~\ref{lem:smoothing}\,(i) that
the inclusion
\begin{equation*}\label{eq:powershift}
  \ran\bigl( (A_{\rm free}-\overline\mu)^{-1}(A_{\rm free}-\lambda)^{-k}\bigr) \subset H^{2k+2}_\Sigma(\dR^n)
\end{equation*}
holds for all $k=0,1,\dots,m$. Moreover, since by Proposition~\ref{prop:qbt1} we have
$A_{\rm free}=\wt T \upharpoonright\ker\wt\Gamma_0$, Proposition~\ref{gammaprop}\,(ii)
implies that
\begin{equation*}
  \wt\gamma(\mu)^*(A_{\rm free} - \lambda)^{-k}
  = \wt\Gamma_1(A_{\rm free}-\overline\mu)^{-1}(A_{\rm free} - \lambda)^{-k}
\end{equation*}
and hence
\begin{equation}\label{eq:range}
 \ran\bigl(\wt\gamma(\mu)^*(A_{\rm free} - \lambda)^{-k}\bigr)\subset H^{2k+3/2}(\Sigma)
\end{equation}
by the properties of the trace map $\wt\Gamma_1$, cf.\ \eqref{Sobtrace}.
Now the  estimates in (a) and (b) follow from \eqref{eq:range} and Lemma~\ref{le.s_emb} with
$\cK = L^2(\dR^n)$, $q_2 = 2k + \frac32$ and with $q_1 = 0$ for (a) and $q_1=1$ for (b), respectively.
The estimate in (c) follows from (a) by taking the adjoint.
\end{proof}

% *******************************************************************
\subsection{Resolvent power differences for the pairs $\{A_{\delta,\alpha}, A_{\rm free}\}$,
$\{A_{\delta^\prime\!,\beta}, A_{\rm free}\}$ and $\{A_{\delta'\!,\beta},A_{\rm N,i,e}\}$}
\label{sec:pwrdiff}
% *******************************************************************

In the next theorem we prove $\sS_{p,\infty}$-properties of resolvent power differences
for the self-adjoint operators $A_{\delta,\alpha}$ and $A_{\rm free}$.
The theorem and its corollary are parts of Theorems~C and D in the introduction. 
%We should
%stress that the formulation given below is a bit stronger than the one in the introduction. We use
%the scale of weak Schatten--von Neumann classes instead of usual Schatten--von Neumann classes.
%This allows us to give a finer estimate of the decay rate of the singular values.

%-------------A_{\delta,\alpha}--------A_{\rm free}--------------------
\begin{thm}\label{thm:S_infty2}
Let $\alpha\in L^\infty(\Sigma)$ be a real-valued function on $\Sigma$, and
let $A_{\delta,\alpha}$ and $A_{\rm free}$
be the self-adjoint operators defined in~\eqref{eq:delta} and \eqref{Afree}, respectively.
Assume that $V\in W^{2m-2, \infty}_{\Sigma}(\dR^n)$ for some $m\in\dN$.
Then
\begin{equation*}
  (A_{\delta,\alpha}-\lambda)^{-l} - (A_{\rm free}-\lambda)^{-l}
  \in \sS_{\frac{n-1}{2l+1},\infty}\bigl(L^2(\dR^n)\bigr)
\end{equation*}
for all $l=1,2,\dots,m$ and for all $\lambda\in\rho(A_{\delta,\alpha})\cap\rho(A_{\rm free})$.
\end{thm}

\begin{proof}
We prove the theorem by applying Lemma~\ref{lem:respwrdiff}.
Fix an arbitrary $\lambda_0\in\dC\setminus\dR$, and let $\wt\gamma$, $\wt M$ be as in
Proposition~\ref{prop:qbt1}.
By Theorem~\ref{thm:delta} the resolvent difference of $A_{\delta,\alpha}$ and $A_{\rm free}$ at the point $\lambda_0$
can be written in the form
\begin{equation*}
  (A_{\delta,\alpha}-\lambda_0)^{-1} - (A_{\rm free}-\lambda_0)^{-1}
  = \wt\gamma(\lambda_0)\bigl(I - \alpha\wt M(\lambda_0) \bigr)^{-1}\alpha \wt\gamma(\ov\lambda_0)^*,
\end{equation*}
where $(I - \alpha\wt M(\lambda_0))^{-1}\alpha\in\cB(L^2(\Sigma))$.
Proposition~\ref{prop:prelim}\,(i)\,(a) and (c) imply
that the assumptions in Lemma~\ref{lem:respwrdiff} are satisfied with
\begin{align*}
  & H = A_{\delta,\alpha}, \quad K = A_{\rm free}, \quad
  B = \wt\gamma(\lambda_0), \quad
  C = \bigl(I - \alpha\wt M(\lambda_0)\bigr)^{-1}\alpha\wt\gamma(\ov\lambda_0)^*, \\
  & a = \frac{2}{n-1}\,, \quad b_1 = b_2 = \frac{3/2}{n-1}\,, \quad r=m.
\end{align*}
Since $b=b_1+b_2-a = \frac{1}{n-1}$, Lemma~\ref{lem:respwrdiff} implies the assertion of the theorem.
\end{proof}

The previous theorem has a direct application in mathematical scattering theory.
Consider the pair $\{A_{\delta,\alpha},A_{\rm free}\}$ of self-adjoint operators
as a scattering system; here $A_{\rm free}$ stands for the unperturbed operators and $A_{\delta,\alpha}$
is singularly perturbed by a $\delta$-potential of strength $\alpha$ supported on the hypersurface $\Sigma$.
It is well known (see, e.g.\ \cite[Theorem~X.4.8]{K95}) that if, for some $m\in\dN$, the 
difference of the $m$th powers of the resolvents of $A_{\delta,\alpha}$ and $A_{\rm free}$
is a trace class operator, i.e.\ if
\[
  (A_{\delta,\alpha}-\lambda_0)^{-m}-(A_{\rm free}-\lambda_0)^{-m} \in \frS_1
\]
for some $\lambda_0\in\rho(A_{\delta,\alpha})\cap\rho(A_{\rm free})$,
then the corresponding wave operators
\begin{equation*}
  W_\pm(A_{\delta,\alpha},A_{\rm free}) \defeq \strlim\limits_{t\rightarrow \pm \infty}
  e^{itA_{\delta,\alpha}}e^{-itA_{\rm free}}P_{\rm ac}(A_{\rm free})
\end{equation*}
exist and are complete, i.e.\ the strong limit exists everywhere
and the ranges coincide with the absolutely continuous subspace of the perturbed
operator $A_{\delta,\alpha}$. Here $P_{\rm ac}(A_{\rm free})$ denotes the
orthogonal projection onto the absolutely continuous subspace of the unperturbed operator $A_{\rm free}$.
This implies, in particular, that the absolutely continuous parts of $A_{\delta,\alpha}$
and $A_{\rm free}$ are unitarily equivalent and that the absolutely continuous spectra coincide:
$\sigma_{\rm ac}(A_{\delta,\alpha})=\sigma_{\rm ac}(A_{\rm free})$,
cf.\ \cite[Theorem~X.4.12, Remark~X.4.13]{K95} and \cite{RS79-III, Y92}.

The next corollary shows that for sufficiently smooth potentials $V$ the wave operators
of the scattering system $\{A_{\delta,\alpha},A_{\rm free}\}$ exist in any space dimension.

% ------------------------------------------------------------------
\begin{cor}\label{cor1}
Let the assumptions be as in Theorem~\ref{thm:S_infty2}. If $V\in W^{k, \infty}_{\Sigma}(\dR^n)$
for some even $k$ and $k>n-4$, then the wave operators $W_\pm(A_{\delta,\alpha},A_{\rm free})$
exist and are complete, and hence the absolutely continuous parts of
$A_{\delta,\alpha}$ and $A_{\rm free}$ are unitarily equivalent.

In particular, if $V=0$, then $W_\pm(A_{\delta,\alpha},A_{\rm free})$ exist and are complete
for any $n\ge2$ and $\sigma_{\rm ac}(A_{\delta,\alpha})=[0,\infty)$.
\end{cor}

In the next theorem we prove $\sS_{p,\infty}$-properties for resolvent power differences
of the self-adjoint operators $A_{\delta'\!,\beta}$ and $A_{\rm free}$.
The theorem and its corollary are the second parts of Theorems~C and D in the introduction.
The formulation given below is a bit stronger than the one in the introduction.

% -------------------------------------------------------------------
\begin{thm}
\label{thm:S_infty4}
Let $\beta$ be a real-valued function on $\Sigma$ such that $1/\beta\in L^\infty(\Sigma)$, and
let $A_{\delta'\!,\beta}$ and $A_{\rm free}$
be the self-adjoint operators defined in~\eqref{eq:delta'} and \eqref{Afree}, respectively.
Assume that $V\in W^{2m-2,\infty}_{\Sigma}(\dR^n\backslash\Sigma)$ for some $m\in\dN$.
Then
\begin{equation*}
  (A_{\delta'\!,\beta} - \lambda)^{-l}-(A_{\rm free} - \lambda)^{-l}
  \in \sS_{\frac{n-1}{2l},\infty}\bigl(L^2(\dR^n)\bigr)
\end{equation*}
for all $l=1,2,\dots,m$ and for all $\lambda\in\rho(A_{\rm \delta'\!,\beta})\cap\rho(A_{\rm free})$.
\end{thm}

\begin{proof}
First we apply Lemma~\ref{lem:respwrdiff} to the difference of the $l$th powers of the
resolvents of $A_{\rm free}$ and $A_{\rm N,i,e}$.
Fix an arbitrary $\lambda_0\in\dC\setminus\dR$ and let $\wh\gamma$ and $\wh M$ be the $\gamma$-field and
Weyl function associated with the quasi boundary triple in Proposition~\ref{prop:qbt2}.
Since the operators $A_{\rm free}$ and $A_{\rm N,i,e}$ are both self-adjoint, in analogy
to \eqref{krein1} we have
\begin{equation*}
\label{eq:0}
  (A_{\rm free}-\lambda_0)^{-1} - (A_{\rm N,i,e}-\lambda_0)^{-1}
  = -\wh\gamma(\lambda_0)\wh M(\lambda_0)^{-1}\wh\gamma(\ov\lambda_0)^*;
\end{equation*}
see \cite[Corollary 3.11]{BLL10}.
Furthermore, by Proposition~\ref{gammaprop}\,(v) and Proposition~\ref{prop:qbt2}\,(iii)
the operator $\wh M(\lambda_0)$ is bijective and closed as an operator from $L^2(\Sigma)$
onto $H^1(\Sigma)$. Hence $\dom \wh M(\lambda_0)^{-1}=H^1(\Sigma)$
and  $\wh M(\lambda_0)^{-1}$ is closed as an operator from
$H^1(\Sigma)$ into $L^2(\Sigma)$. Thus, we can conclude that $\wh M(\lambda_0)^{-1}\in \cB(H^1(\Sigma), L^2(\Sigma))$.
Set
\[
  H \defeq A_{\rm free}, \quad K \defeq A_{\rm N,i,e}, \quad
  B \defeq -\wh\gamma(\lambda_0), \quad C \defeq \wh M(\lambda_0)^{-1}\wh\gamma(\ov\lambda_0)^*.
\]
Then Proposition~\ref{prop:prelim}\,(ii)\,(b) and (c) imply that the assumptions
in Lemma~\ref{lem:respwrdiff} are satisfied with
\[
  a = \frac{2}{n-1}\,, \quad b_1 = \frac{3/2}{n-1}\,, \quad b_2 = \frac{1/2}{n-1}\,,\quad r=m.
\]
Since $b=b_1+b_2-a=0$, Lemma~\ref{lem:respwrdiff} implies that
\begin{equation}\label{Sinfty}
  (A_{\rm free} - \lambda)^{-l} - (A_{\rm N,i,e} - \lambda)^{-l}
  \in \sS_{\frac{n-1}{2l},\infty}\bigl(L^2(\dR^n)\bigr)
\end{equation}
for all $\lambda\in\rho(A_{\rm free})\cap\rho(A_{\rm N,i,e})$ and all $l=1,2,\dots,m$.
This observation together with Theorem~\ref{thm:S_infty3} shows that
\begin{equation}\label{ztr}
  (A_{\delta'\!,\beta} - \lambda)^{-l}-(A_{\rm free} - \lambda)^{-l}
  \in \sS_{\frac{n-1}{2l},\infty}\bigl(L^2(\dR^n)\bigr)
\end{equation}
for all $l=1,2,\dots,m$ and for all $\lambda\in\rho(A_{\rm \delta'\!,\beta})
\cap\rho(A_{\rm free})\cap\rho(A_{\rm N,i,e})$.
As the resolvent power difference in \eqref{ztr} is analytic in $\lambda$, it
follows that \eqref{ztr} holds also for those points
$\lambda\in\rho(A_{\rm \delta'\!,\beta})\cap\rho(A_{\rm free})$ which are isolated
eigenvalues of $A_{\rm N,i,e}$; note that we know already that the essential spectra
of $A_{\rm N,i,e}$, $A_{\delta'\!,\beta}$ and $A_{\rm free}$ coincide because the
relations \eqref{Sinfty} and \eqref{ztr} are true at least for non-real $\lambda$.
\end{proof}

\begin{rem}
We note that in \eqref{Sinfty} it is shown
that the difference of the $m$th powers of the resolvents of $A_{\rm free}$ and $A_{\rm N,i,e}$ belongs to
the class $\sS_{\frac{n-1}{2m},\infty}$ provided  $V\in W^{2m-2,\infty}_\Sigma(\dR^n\backslash\Sigma)$.
This is a slight improvement of a result in \cite{B62}. For infinitely smooth $V$ the asymptotics of the singular values
have been studied in \cite{BS79,BS80} and \cite{G84}.
\end{rem}

The following corollary is the counterpart of Corollary~\ref{cor1} for the
scattering system $\{A_{\delta^\prime\!,\beta},A_{\rm free}\}$.

% -------------------------------------------------------------------
\begin{cor}\label{cor2}
Let the assumptions be as in Theorem~\ref{thm:S_infty4}. If $V\in W^{k, \infty}_{\Sigma}(\dR^n\backslash\Sigma)$
for some even $k$ and $k>n-3$, then the wave operators $W_\pm(A_{\delta^\prime\!,\beta},A_{\rm free})$
exist and are complete, and hence the absolutely continuous parts of
$A_{\delta^\prime\!,\beta}$ and $A_{\rm free}$ are unitarily equivalent.

In particular, if $V=0$, then $W_\pm(A_{\delta'\!,\beta},A_{\rm free})$ exist and are complete
for any $n\ge2$ and $\sigma_{\rm ac}(A_{\delta'\!,\beta})=[0,\infty)$.
\end{cor}

The next result on the $\sS_{p,\infty}$-properties of the resolvent power differences
of $A_{\delta'\!,\beta}$ and $A_{\rm N,i,e}$ completes the proof of Theorem~\ref{thm:S_infty4},
but is also of independent interest. We do not formulate the corresponding
corollary for the scattering system $\{A_{\delta'\!,\beta},A_{\rm N,i,e}\}$.

% -------------------------------------------------------------------
\begin{thm}\label{thm:S_infty3}
Let $\beta$ be a real-valued function on $\Sigma$ such that $1/\beta\in L^\infty(\Sigma)$, and
let $A_{\delta'\!,\beta}$ and $A_{\rm N,i,e}$
be the self-adjoint operators defined in~\eqref{eq:delta'} and \eqref{ADANoplus2}, respectively.
Assume that $V\in W^{2m-2,\infty}_{\Sigma}(\dR^n\backslash\Sigma)$ for some $m\in\dN$. Then
\begin{equation*}
  (A_{\delta'\!,\beta} - \lambda)^{-l}-(A_{\rm N,i,e} - \lambda)^{-l}
  \in \sS_{\frac{n-1}{2l+1},\infty}\bigl(L^2(\dR^n)\bigr)
\end{equation*}
for all $l = 1,2,\dots,m$ and all $\lambda\in\rho(A_{\delta'\!,\beta})\cap\rho(A_{\rm N,i,e})$.
\end{thm}

\begin{proof}
As in the proofs of Theorem~\ref{thm:S_infty2} and \ref{thm:S_infty4}
fix $\lambda_0\in\dC\setminus\dR$ and let $\wh\gamma$ and $\wh M$ be the $\gamma$-field and
Weyl function associated with the quasi boundary triple in Proposition~\ref{prop:qbt2}.
By Theorem~\ref{thm:delta'} the resolvent difference of $A_{\delta^\prime\!,\beta}$ and $A_{\rm N,i,e}$
at the point $\lambda_0$ can be written in the form
\begin{equation*}
  (A_{\delta^\prime\!,\beta}-\lambda_0)^{-1} - (A_{\rm N,i,e}-\lambda_0)^{-1}
  = \wh\gamma(\lambda_0)\bigl(I - \beta^{-1}\wh M(\lambda_0)\bigr)^{-1}\beta^{-1}\wh\gamma(\ov\lambda_0)^*,
\end{equation*}
where $(I - \beta^{-1}\wh M(\lambda_0))^{-1}\beta^{-1}\in\cB(L^2(\Sigma))$.
Proposition~\ref{prop:prelim}\,(ii)\,(a) and (c) imply that the assumptions in
Lemma~\ref{lem:respwrdiff} are satisfied with
\begin{alignat*}{2}
  H &= A_{\delta'\!,\beta}, \quad & K &= A_{\rm N,i,e}, \quad
  B = \wh\gamma(\lambda_0), \quad
  C = \bigl(I - \beta^{-1}\wh M(\lambda_0)\bigr)^{-1}\beta^{-1}\wh\gamma(\ov\lambda_0)^*, \\
  a &= \frac{2}{n-1}\,, \quad & b_1 &= b_2 = \frac{3/2}{n-1}\,, \quad r=m.
\end{alignat*}
Since $b=b_1+b_2-a = \frac{1}{n-1}$, Lemma~\ref{lem:respwrdiff} implies the assertion of the theorem.
\end{proof}

% *******************************************************************
\subsection*{Acknowledgements}
The authors gratefully acknowledge the hospitality and stimulating working atmosphere 
at the Isaac Newton Institute for Mathematical Sciences (Cambridge, UK).
% *******************************************************************

% *******************************************************************

\end{document}